\documentclass[preprint,11pt]{elsarticle}

\usepackage[T1]{fontenc}
\usepackage[utf8]{inputenc}
\usepackage{lmodern}
\usepackage{amsmath,amssymb,amsthm}
\usepackage[a4paper,left=2.35cm,right=2.35cm,top=2.25cm,bottom=2.45cm]{geometry}
\usepackage{enumitem}
\usepackage{aliascnt}
\usepackage{etoolbox}
\usepackage[colorlinks=true,linkcolor=blue,citecolor=blue,urlcolor=blue]{hyperref}
\hypersetup{pdfauthor={Bohan Yang}}
\usepackage[nameinlink,capitalize]{cleveref}

\makeatletter
\patchcmd{\pprintMaketitle}{\normalsize\elsauthors\par\vskip10pt
      \footnotesize\itshape\elsaddress\par\vskip36pt}{\normalsize\elsauthors\par\vskip2pt
      \footnotesize\itshape\elsaddress\par\vskip10pt}{}{}
\patchcmd{\pprintMaketitle}{%
    \hrule\vskip12pt
    \ifvoid\absbox\else\unvbox\absbox\par\vskip10pt\fi
    \ifvoid\keybox\else\unvbox\keybox\par\vskip10pt\fi
    \hrule\vskip12pt
}{%
    \vskip12pt
    \ifvoid\absbox\else\unvbox\absbox\par\vskip10pt\fi
    \ifvoid\keybox\else\unvbox\keybox\par\vskip10pt\fi
    \vskip12pt
}{}{\PackageError{SDW}{Unable to remove the abstract rules}{}}

\renewenvironment{abstract}{%
  \global\setbox\absbox=\vbox\bgroup
  \hsize=\textwidth
  \hbox to \textwidth\bgroup\hfil
  \begin{minipage}{0.82\textwidth}
  \small
  \begin{center}\textbf{Abstract}\end{center}
  \vspace{-0.4\baselineskip}
  \noindent\ignorespaces
}{%
  \par\end{minipage}\hfil\egroup
  \vspace{0.7\baselineskip}
  \egroup
}

\makeatother

\numberwithin{equation}{section}

\newcommand{\R}{\mathbb R}
\newcommand{\Z}{\mathbb Z}
\newcommand{\N}{\mathbb N}
\newcommand{\SL}{\operatorname{SL}}
\newcommand{\diag}{\operatorname{diag}}
\newcommand{\Sing}{\operatorname{Sing}}
\newcommand{\cL}{\mathcal L}
\newcommand{\cK}{\mathcal K}
\newcommand{\cT}{\mathcal T}
\newcommand{\cF}{\mathcal F}
\newcommand{\cA}{\mathcal A}
\newcommand{\cB}{\mathcal B}
\newcommand{\eps}{\varepsilon}
\newcommand{\wh}{\widehat}
\newcommand{\abs}[1]{\lvert #1\rvert}
\newcommand{\norm}[1]{\lVert #1\rVert}
\newcommand{\card}[1]{\# #1}
\DeclareMathOperator{\dist}{dist}
\DeclareMathOperator{\diam}{diam}
\DeclareMathOperator{\vol}{vol}

\theoremstyle{plain}
\newtheorem{theorem}{Theorem}[section]
\newaliascnt{lemma}{theorem}
\newtheorem{lemma}[lemma]{Lemma}
\aliascntresetthe{lemma}
\newaliascnt{proposition}{theorem}
\newtheorem{proposition}[proposition]{Proposition}
\aliascntresetthe{proposition}
\newaliascnt{corollary}{theorem}
\newtheorem{corollary}[corollary]{Corollary}
\aliascntresetthe{corollary}
\theoremstyle{remark}
\newtheorem*{remark}{Remark}

\begin{document}
\raggedbottom

\begin{frontmatter}

\title{Slowly Divergent Trajectories for Weighted Singular Vectors in \texorpdfstring{$\mathbb R^2$}{R2}}

\author[simis]{Bohan Yang\corref{corresponding-author}}
\ead{bhyang@simis.cn}
\cortext[corresponding-author]{Corresponding author.}
\affiliation[simis]{%
  organization={Shanghai Institute for Mathematics and Interdisciplinary Sciences},
  city={Shanghai},
  postcode={200433},
  country={P. R. China}}

\begin{abstract}
Let $w=(w_1,w_2)$ satisfy $w_1>w_2>0$ and $w_1+w_2=1$.  For every
function $f:[0,\infty)\to(0,\infty)$ with $f(t)\to0$, we prove that the
set of $w$-singular vectors whose weighted shortest-vector function
satisfies $W_x(t)\ge\log f(t)$ for all sufficiently large $t$
has Hausdorff dimension
$s_w=2-(1+w_1)^{-1}$, equal to the full Hausdorff dimension of
$\operatorname{Sing}_w(2)$.  For every $0<\nu<1$ and $\mu>0$, a
separate power schedule gives Hausdorff dimension $s_w$ for the
weighted uniform approximation set with rate
$Q^{-1}\exp\!\bigl(-\mu(\log Q)^\nu\bigr)$, whereas the lower-envelope
theorem gives the same dimension for its complement in
$\operatorname{Sing}_w(2)$.  We give a direct proof of the
lower-envelope theorem by adapting the self-affine construction of
Liao--Shi--Solan--Tamam to a variable sequence of return times and
using an empty-denominator-window argument to control intermediate
cusp excursions without further pruning the tree.
\end{abstract}

\end{frontmatter}

\section{Introduction}

Fix
\begin{equation}\label{eq:weight}
w=(w_1,w_2),\qquad w_1>w_2>0,\qquad w_1+w_2=1.
\end{equation}
Here \(\N=\{1,2,\ldots\}\).  Following Liao, Shi, Solan, and
Tamam (LSST) \cite{LSST}, a vector \(x=(x_1,x_2)\in\R^2\) is
\(w\)-singular if, for every \(\eps>0\) and all sufficiently large
\(Q\), there exists \((p,q)\in\Z^2\times\N\) such that
\begin{equation}
 |qx_i-p_i|<\eps^{w_i}Q^{-w_i}\quad(i=1,2),\qquad q<Q.
\end{equation}
Write \(\Sing_w(2)\) for the set of such vectors.  Let
\(\cL_3=\SL_3(\R)/\SL_3(\Z)\), put
\[
 a_t=\diag(e^{w_1t},e^{w_2t},e^{-t}),
 \qquad
 u_x=
 \begin{pmatrix}1&0&-x_1\\0&1&-x_2\\0&0&1\end{pmatrix},
\]
and define, for \(y=(y_1,y_2,y_3)\in\R^3\) and
\(\Lambda\in\cL_3\),
\[
 \norm{y}_w
 =\max\{\abs{y_1}^{1/w_1},\abs{y_2}^{1/w_2},\abs{y_3}\},
 \qquad
 \norm{\Lambda}_w
 =\min_{y\in\Lambda\setminus\{0\}}\norm{y}_w.
\]
The double bars in \(\norm{\cdot}_w\) are only notation: this is an
anisotropic size functional adapted to \(a_t\), not a norm.
Throughout the paper, \(\norm{\cdot}\) denotes the Euclidean norm.
Unless otherwise stated, \(A\ll B\) means \(A\le CB\) for a constant
\(C>0\) depending only on the fixed weight \(w\).
The weighted shortest-vector function of \(x\) is
\begin{equation}\label{eq:W}
 W_x(t)=\log\norm{a_tu_x\Z^3}_w.
\end{equation}
Equivalently,
\[
 e^{W_x(t)}
 =\min_{(p,q)\in\Z^2\times\N}
 \max\left\{
 e^t\abs{qx_1-p_1}^{1/w_1},
 e^t\abs{qx_2-p_2}^{1/w_2},
 e^{-t}q
 \right\}.
\]
Dani's correspondence \cite{Dani85} gives
\[
 x\in\Sing_w(2)
 \Longleftrightarrow
 a_tu_x\Z^3\text{ diverges in }\cL_3
 \Longleftrightarrow
 W_x(t)\longrightarrow-\infty.
\]
Thus a lower bound on \(W_x\) prescribes how slowly a divergent
trajectory enters the cusp.

Write \(\Sing(d)\) for the unweighted singular set in \(\R^d\).
The first exact dimension result beyond dimension one is due to Cheung,
who proved \cite[Theorem~1.1]{Cheung11}
\[
 \dim_H\Sing(2)=\frac43.
\]
In the same work, he constructed divergent trajectories with arbitrarily
prescribed slow escape, answering a question of Starkov
\cite[p.~213]{Starkov}; see \cite[Theorem~1.4]{Cheung11}.
This dynamical result is related to his earlier construction of arbitrarily
slowly divergent Teichm\"uller geodesics in moduli space \cite{Cheung04}.
Cheung and Chevallier subsequently extended the exact dimension formula to
every \(d\ge2\) \cite[Theorem~1.1]{CheungChevallier}:
\[
 \dim_H\Sing(d)=\frac{d^2}{d+1}.
\]
Liao, Shi, Solan, and Tamam adapted Cheung's construction to planar
weighted approximation and proved \cite[Theorem~1.1]{LSST}
\begin{equation}\label{eq:LSST-dimension}
 \dim_H\Sing_w(2)=s_w:=2-\frac1{1+w_1}.
\end{equation}
More generally, let \(d\ge2\), and let
\(\boldsymbol\omega=(\omega_1,\ldots,\omega_d)\) be an ordered positive
weight.  Write \(\Sing_{\boldsymbol\omega}(d)\) for the corresponding
weighted singular set in \(\R^d\).  The results of Kim--Park
\cite[Theorem~1.1]{KimPark} and
Aggarwal--Ghosh
\cite[Theorem~1.3 and Remark~1.4]{AggarwalGhosh} give
\[
 d-\frac1{1+\omega_1}
 \le \dim_H\Sing_{\boldsymbol\omega}(d)
 \le \dim_P\Sing_{\boldsymbol\omega}(d)
 \le d-\frac{d\omega_d}{1+\omega_1}.
\]
Here \(\omega_1\ge\cdots\ge\omega_d>0\),
\(\sum_i\omega_i=1\), and \(\dim_P\) denotes packing dimension.

These dimension results impose no rate constraint.  Kleinbock,
Moshchevitin, Warren, and Weiss (KMWW) developed the weighted uniform
approximation framework used below and studied existence, intersection
properties, and attainable rates
\cite{KleinbockMoshchevitinWarrenWeiss}.  For the standard unweighted
\(m\times n\) matrix flow, Das, Fishman, Simmons, and Urba\'nski
(DFSU) proved the corresponding arbitrary lower-envelope result
\cite[Theorem~3.12]{DFSU}.  In the relevant \(2\times1\) case, the two
expanding coordinates have the same rate.  By contrast, under \(a_t\)
they expand at the unequal rates \(w_1\) and \(w_2\).  Thus, when
\(w_1>w_2\), \(a_t\) is not a scalar time change of the standard
\(2\times1\) matrix flow, and the DFSU theorem does not directly imply
\cref{thm:main}.  Our proof is direct and does not invoke the DFSU
variational principle.  Solan's general-flow variational principle
gives dimension estimates for divergent sets under arbitrary diagonal
flows \cite[Corollary~2.34 and Example~3.1]{SolanGeneralFlow}, but not
the lower-envelope dimension of the two-dimensional family
\(\{u_x\Z^3:x\in\R^2\}\).  To our knowledge, it was not previously
known whether imposing an arbitrary eventual lower envelope preserves
the full Hausdorff dimension of \(\Sing_w(2)\) when \(w_1>w_2\).

For \(f:[0,\infty)\to(0,\infty)\) with \(f(t)\to0\), let
\[
 E_{f,w}
 =\left\{x\in\Sing_w(2):
 W_x(t)\ge\log f(t)
 \text{ for all sufficiently large }t
 \right\}.
\]

\begin{theorem}\label{thm:main}
For every such \(f\),
\[
 \dim_H E_{f,w}=s_w.
\]
\end{theorem}

We call \(x\in\R^2\) \(w\)-very well approximable if there exists
\(\delta>0\) such that
\[
 \max_{i=1,2}\abs{qx_i-p_i}^{1/w_i}<q^{-1-\delta}
\]
for infinitely many \((p,q)\in\Z^2\times\N\).
If \(-\log f(t)=o(t)\), the balance-time criterion in Section~2
shows that every \(x\in E_{f,w}\) is singular but not
\(w\)-very well approximable.  Hence
\[
 \dim_H\{x\in\Sing_w(2):x\text{ is not }w\text{-very well
 approximable}\}=s_w.
\]
The lower bound in \cref{thm:main} has two parts.  First, we show that
the LSST tree remains stable under level-dependent return times and
scales, with the same branching exponent.  Second, denominator
localization turns every intermediate excursion below the prescribed
envelope into an explicit interval for the denominator.  The return
times are chosen so that this interval is empty, and no further
pruning of the tree is required.

A complementary result concerns prescribed rates in weighted uniform
approximation.  We use the notation \(\operatorname{UA}_w(\Phi)\),
defined precisely in \cref{sec:uniform-approximation} as a
specialization of the KMWW framework.

\begin{theorem}\label{thm:subpower-uniform-rate}
Let \(0<\nu<1\) and \(\mu>0\).  Let
\(\Phi_{\nu,\mu}:[1,\infty)\to(0,\infty)\) be positive and
nonincreasing, with
\[
 \Phi_{\nu,\mu}(Q)
 =Q^{-1}\exp\!\left(-\mu(\log Q)^\nu\right)
 \qquad(Q\ge e).
\]
Then
\[
 \dim_H\operatorname{UA}_w(\Phi_{\nu,\mu})
 =\dim_H\left(
 \Sing_w(2)\setminus\operatorname{UA}_w(\Phi_{\nu,\mu})
 \right)
 =s_w.
\]
\end{theorem}

The two dimension equalities have different proofs.  A separate power
schedule in the LSST tree gives the first, while \cref{thm:main} and a
general rate-avoidance argument give the second.  Thus both the vectors
attaining this prescribed subpower improvement of the Dirichlet rate
and the singular vectors failing to attain it have Hausdorff dimension
\(s_w\).

Returning to the lower-envelope theorem, standard slicing and
local-product arguments give the following full-lattice consequence.

\begin{corollary}\label{cor:homogeneous}
With respect to any Riemannian metric on \(\cL_3\),
\[
 \dim_H
 \left\{
 \Lambda\in\cL_3:
 \norm{a_t\Lambda}_w\longrightarrow0,\quad
 \norm{a_t\Lambda}_w\ge f(t)\ \text{eventually}
 \right\}
 =8-\frac1{1+w_1}.
\]
\end{corollary}

For \(s\in\R\), put
\[
 v_s=
 \begin{pmatrix}1&s&0\\0&1&0\\0&0&1\end{pmatrix}.
\]
On compact \(s\)-intervals, the lattices \(v_s\Z^3\) satisfy the
initial-lattice hypotheses of the variable-step construction
uniformly.  Apply that construction on the two-dimensional
\(u_x\)-slices to a regularization of
\(t\mapsto c f(t)^{w_2}\), for a suitable \(c>0\).  Bounded
conjugation in the five nonexpanding directions and the comparison of
the Euclidean and weighted minima convert this lower envelope into
\(f\).  Slicing in \(s\) and taking the local product therefore adds
six dimensions, as in \cite[Corollary~1.2]{Cheung11},
\cite[Theorem~1.5 and Corollary~1.6]{LSST}, and
\cite[Theorem~1.2 and Corollary~1.3]{KimPark}.  Solan's estimate
\cite[Corollary~2.34 and Example~3.1]{SolanGeneralFlow} gives the
matching upper bound.  We omit these standard slicing and local-product
steps.

The proof below is written for \(w_1>w_2\), because its use of
\cref{lem:LSST-3.10} requires the positive root gap \(w_1-w_2\).
The unweighted case \(w_1=w_2=1/2\) would require replacing this step
with Cheung's unweighted counting argument.  After the elementary
comparison between the Euclidean and weighted minima, the analogous
lower-envelope statement follows from
\cite[Theorem~3.12]{DFSU}, so we do not carry out that separate
argument.

\emph{Organization of the paper.}
Section~2 develops the balance-time criterion and regularizes the
target.  Section~3 constructs and analyzes the variable-step tree.
Section~4 combines denominator localization with the global schedule,
and Section~5 proves \cref{thm:main}.
\Cref{sec:uniform-approximation} introduces the weighted uniform
approximation notation, proves
\cref{thm:subpower-uniform-rate}, and derives the general
rate-avoidance consequence of \cref{thm:main}.

\section{Balance functions and target regularization}

We formulate the balance-time criterion for an arbitrary base lattice;
the standard lattice is the case needed for \cref{thm:main}.  A
nonzero vector \(v\) in a lattice \(\Lambda\) is called
\emph{primitive} if \(\Lambda\cap\R v=\Z v\).  We write
\(\wh\Lambda\) for the set of primitive vectors in \(\Lambda\).
Let \(e_3=(0,0,1)\).

Let \(\Lambda\in\cL_3\), let \(x\in\R^2\), and, for
\(v=(p_1,p_2,q)\in\Lambda\), put
\begin{equation}
 R_x(v)=
 \max_{i=1,2}\abs{qx_i-p_i}^{1/w_i}.
\end{equation}
When \(q>0\) and \(R_x(v)>0\), define the
\emph{balance time} \(\tau_x(v)\) and the
\emph{balanced minimum} \(m_x(v)\) by
\begin{equation}
 \tau_x(v)=
 \frac12\bigl(\log q-\log R_x(v)\bigr),
 \qquad
 m_x(v)=
 \frac12\bigl(\log q+\log R_x(v)\bigr).
\end{equation}
Then
\begin{equation}\label{eq:balance-identity}
 \log\max\left\{
 e^t\abs{qx_1-p_1}^{1/w_1},
 e^t\abs{qx_2-p_2}^{1/w_2},
 e^{-t}q
 \right\}
 =m_x(v)+\abs{t-\tau_x(v)}.
\end{equation}
If \(v=kv_0\) for an integer \(k\ge2\), then
the expression inside the logarithm in
\eqref{eq:balance-identity} is at least \(k\) times the corresponding
expression for \(v_0\).  Thus the minimum may be taken over primitive
vectors.  In particular, if
\(u_x\Z^3\cap\R e_3=\{0\}\), then the definition
\eqref{eq:W} has the simplified lower-envelope form
\begin{equation}\label{eq:W-lower-envelope}
 W_x(t)
 =\min_{\substack{v=(p_1,p_2,q)\in\wh{\Z^3}\\q>0}}
 \bigl(m_x(v)+\abs{t-\tau_x(v)}\bigr),
 \qquad t\ge0.
\end{equation}

For a general base lattice, the contributions from vectors with zero
third coordinate tend to infinity.  Hence, if
\(u_x\Lambda\cap\R e_3=\{0\}\), the analogue of
\eqref{eq:W-lower-envelope} holds for all sufficiently large \(t\)
when the minimum in \eqref{eq:balance-identity} is taken over
\(\wh\Lambda\).

\begin{proposition}\label{prop:eventual-balance}
Let \(\Lambda\in\cL_3\), let \(x\in\R^2\) satisfy
\(u_x\Lambda\cap\R e_3=\{0\}\), and let
\(\psi:[0,\infty)\to[0,\infty)\) be nondecreasing and
\(1\)-Lipschitz.  The following conditions are equivalent:
\begin{enumerate}[label=\textup{(\roman*)}]
\item for all sufficiently large \(t\),
\[
 \log\min_{\substack{v=(p_1,p_2,q)\in\Lambda\setminus\{0\}}}
 \max\left\{
 e^t\abs{qx_1-p_1}^{1/w_1},
 e^t\abs{qx_2-p_2}^{1/w_2},
 e^{-t}\abs{q}
 \right\}
 \ge-\psi(t);
\]
\item there exists \(T_0\ge0\) such that every
\(v=(p_1,p_2,q)\in\wh\Lambda\) with \(q>0\) satisfies
\[
 \tau_x(v)\ge T_0
 \quad\Longrightarrow\quad
 m_x(v)\ge-\psi(\tau_x(v)).
\]
\end{enumerate}
\end{proposition}

\begin{proof}
If (i) holds, evaluate the contribution of a primitive vector at its
balance time to obtain
\[
 m_x(v)
 \ge
 \log\min_{z\in\Lambda\setminus\{0\}}
 \max\left\{
 e^{\tau_x(v)}R_x(z),
 e^{-\tau_x(v)}\abs{z_3}
 \right\}
 \ge-\psi(\tau_x(v))
\]
whenever \(\tau_x(v)\) is sufficiently large.

Conversely, assume (ii).  Monotonicity and the \(1\)-Lipschitz
property imply
\begin{equation}
 \abs{t-\tau}+\psi(t)\ge\psi(\tau)
 \qquad(t,\tau\ge0).
\end{equation}
Thus every vector with \(\tau_x(v)\ge T_0\) contributes at least
\(-\psi(t)\) at time \(t\).

It remains to treat vectors with bounded balance time.  Their values
\(R_x(v)\) are bounded away from zero.  Otherwise, there would be a
sequence \((v_j)\subset\wh\Lambda\), with \(q_j\) denoting the third
coordinate of \(v_j\), such that
\(R_x(v_j)\to0\) and \(\tau_x(v_j)<T_0\).  The identity
\[
 q_j=R_x(v_j)e^{2\tau_x(v_j)}
\]
would then imply \(u_xv_j\to0\), contrary to the discreteness of
\(u_x\Lambda\).  The corresponding balance functions therefore tend
uniformly to \(+\infty\) as \(t\to\infty\).  The same conclusion holds
for vectors with zero third coordinate, while vectors with negative
third coordinate are handled by changing sign.  Taking the primitive
lower envelope proves (i).
\end{proof}

For the standard lattice, the same identity gives the
very-well-approximable criterion used in the introduction:
\[
 x\text{ is }w\text{-very well approximable}
 \quad\Longleftrightarrow\quad
 \liminf_{t\to\infty}\frac{W_x(t)}t<0.
\]
Indeed, if \(R_x(v)<q^{-1-\delta}\), then at
\(t=\tau_x(v)\),
\[
 \frac{W_x(t)}t
 \le\frac{m_x(v)}{\tau_x(v)}
 \le-\frac{\delta}{2+\delta}.
\]
Conversely, if \(W_x(t_k)\le-\eta t_k\) along a sequence
\(t_k\to\infty\), where \(0<\eta<1\), a minimizing pair satisfies
\[
 q\le e^{(1-\eta)t_k},\qquad
 R_x(v)\le e^{-(1+\eta)t_k}
 \le q^{-(1+\eta)/(1-\eta)}.
\]
If the denominators are bounded, then a fixed pair occurs along a
subsequence and \(R_x(v)=0\), so \(x\) is rational and the conclusion
is immediate.  Otherwise these inequalities give infinitely many
power-law improvements.

\subsection{Regularization}

Since the definition of \(E_{f,w}\) concerns only sufficiently large
\(t\), changing \(f\) on a bounded initial interval does not affect
\(E_{f,w}\).  Choose \(T_f\ge0\) so
that \(f(t)\le1\) for every \(t\ge T_f\), replace \(f\) on
\([0,T_f]\) by \(\min\{1,f(t)\}\), and relabel the resulting function
as \(f\).  Thus \(0<f(t)\le1\) for every \(t\ge0\).
Define
\begin{equation}\label{eq:psi}
f^\downarrow(t)=\sup_{s\ge t}f(s),
\qquad
\psi(t)=
\inf_{0\le s\le t}
\bigl(-\log f^\downarrow(s)+t-s\bigr).
\end{equation}

\begin{lemma}\label{lem:regularization}
The function \(\psi\) is nondecreasing, \(1\)-Lipschitz, and
\(\psi(t)\to\infty\).  Moreover,
\[
0\le\psi(t)\le-\log f^\downarrow(t)
\quad\text{and}\quad
e^{-\psi(t)}\ge f^\downarrow(t)\ge f(t).
\]
Therefore,
\[
\left\{
x\in\Sing_w(2):
W_x(t)\ge-\psi(t)\text{ eventually}
\right\}
\subseteq E_{f,w}.
\]
\end{lemma}

\begin{proof}
The function \(-\log f^\downarrow\) is nonnegative, nondecreasing,
and divergent.  Hence \(0\le\psi(t)\), while choosing \(s=t\) in
\eqref{eq:psi} gives
\(\psi(t)\le-\log f^\downarrow(t)\).  Let \(t_1<t_2\).  Splitting the
defining infimum gives
\[
\psi(t_2)
=
\min\left\{
\psi(t_1)+t_2-t_1,
\inf_{t_1<s\le t_2}
\bigl(-\log f^\downarrow(s)+t_2-s\bigr)
\right\}.
\]
Every term in the second infimum is at least
\(-\log f^\downarrow(t_1)\ge\psi(t_1)\).  Therefore,
\[
0\le\psi(t_2)-\psi(t_1)\le t_2-t_1,
\]
so \(\psi\) is nondecreasing and \(1\)-Lipschitz.

For \(0\le s\le t/2\), the expression in \eqref{eq:psi} is at least
\(t/2\), while for \(t/2\le s\le t\) it is at least
\(-\log f^\downarrow(t/2)\).  Hence
\[
\psi(t)\ge
\min\{t/2,-\log f^\downarrow(t/2)\}
\longrightarrow\infty.
\]
Finally,
\(-\psi(t)\ge\log f^\downarrow(t)\ge\log f(t)\).  Together with
membership in \(\Sing_w(2)\), this proves the last inclusion.
\end{proof}

The lower bound in \cref{thm:main} therefore reduces to the
following problem: for an arbitrary nondecreasing, \(1\)-Lipschitz
function \(\psi\) with \(\psi(t)\to\infty\), construct a set of
Hausdorff dimension \(s_w\) on which
\(W_x(t)\ge-\psi(t)\) eventually.

\section{The variable-step construction}

The self-affine dimension theorem \cref{thm:LSST-2.1} and the
lattice-point estimates
\cref{lem:LSST-3.7,lem:LSST-3.9,lem:LSST-3.10} are quoted from
\cite{LSST}.  The remainder of this section proves that
their tree construction and dimension calculation remain valid
when the return times and scales vary with the level.  These
variable-step statements are needed because the schedule in Section~4
depends on the prescribed target.

For \(R>0\), let
\[
 B_R=\{y\in\R^3:\norm{y}\le R\}.
\]
For a full-rank lattice \(\Lambda\), write
\(\Lambda^*=\{y\in\R^3:y\cdot z\in\Z\text{ for every }z\in\Lambda\}\)
for its dual lattice.
The notation \(\wh\Lambda\) for primitive vectors is as in
Section~2; in particular, \(\wh{\Lambda^*}\) denotes the
primitive vectors of the dual lattice.
For \(\rho>0\), define
\[
\cK_\rho^*(3)
=
\left\{
\Lambda\in\cL_3:
\norm{\varphi}\ge\rho
\text{ for every }\varphi\in\Lambda^*\setminus\{0\}
\right\}.
\]
For every \(\Lambda\in\cL_3\) with
\(\Lambda\cap\R e_3\ne\{0\}\), let \(\ell_3(\Lambda)>0\) be the
unique number such that
\[
 \Lambda\cap\R e_3=\ell_3(\Lambda)\Z e_3.
\]
Define
\[
\cL_3'
=
\left\{
\Lambda\in\cL_3:
\Lambda\cap\R e_3\ne\{0\},\quad
\frac12<\ell_3(\Lambda)\le1
\right\}.
\]

\subsection{Self-affine and lattice-point estimates}

A self-affine structure in \(\R^2\) is a pair \((\cT,\beta)\), where
\(\cT=\bigsqcup_{n\ge0}\cT_n\) is a rooted tree and \(\beta\) maps
each vertex to a nonempty compact axis-parallel rectangle.  It is
regular if every vertex has at least one child, each child rectangle is
contained in its parent rectangle, and the rectangle diameters tend to
zero along every infinite branch.  Its limit set is
\begin{equation}
\cF(\cT,\beta)
=
\left\{
x\in\R^2:
x\in\bigcap_{n\ge0}\beta(\sigma_n)
\text{ for some infinite branch }(\sigma_n)_{n\ge0}
\right\}.
\end{equation}

\begin{theorem}[{\cite[Theorem~2.1]{LSST}}]\label{thm:LSST-2.1}
Let \((\cT,\beta)\) be a regular self-affine structure in
\(\R^2\).  Suppose every rectangle at level \(n\) has short side
\(S_n\) and long side \(L_n\), with \(S_n\le L_n\).
Assume that every level \(n-1\) vertex has at least \(C_n\)
children, \(C_0=1\), and that distinct children of a level \(n\)
vertex have distance at least \(\rho_{n+1}S_n\), where
\(0<\rho_n\le1\).

Put
\[
P_n=\prod_{k=0}^n C_k,
\qquad
D_n
=
\max\{k\ge n:L_k\ge S_n\}.
\]
Define
\[
s_*=
\sup\left\{
s>0:
\frac{
\log\left(
P_nS_n^s\rho_{n+1}^s
\prod_{i=n+1}^{D_n}\rho_iC_i
\right)
}{
\max\{D_n-n,1\}
}
\longrightarrow+\infty
\right\}.
\]
If \(s_*>1\), then
\[
\dim_H\cF(\cT,\beta)\ge s_*.
\]
An empty product is interpreted as \(1\).
\end{theorem}

The cross-scale factor \(\prod_{i=n+1}^{D_n}\rho_iC_i\) is needed
because the rectangles may have different side lengths.

For a symmetric convex body \(K\subset\R^3\), let
\(\lambda_i(K,\Lambda)\) denote its \(i\)-th successive minimum with
respect to \(\Lambda\).  For
\(\mathbf r=(r_1,r_2,r_3)\in(0,\infty)^3\), put
\[
M_{\mathbf r}
=
\{z\in\R^3:\abs{z_i}\le r_i,\ i=1,2,3\}.
\]
We shall use the following unimodular specialization of the
lattice-point estimate in \cite{LSST}.

\begin{lemma}[{\cite[Lemma~3.7]{LSST}}]\label{lem:LSST-3.7}
There exists an absolute constant \(\widetilde c>0\) such that, for
any \(\Lambda\in\cL_3\), if
\[
\lambda_3(M_{\mathbf r},\Lambda)\le\widetilde c,
\qquad
-\lambda_3(M_{\mathbf r},\Lambda)
\log\lambda_1(M_{\mathbf r},\Lambda)
\le\widetilde c,
\]
then
\[
\frac{4}{5\zeta(3)}\vol M_{\mathbf r}
\le
\card{M_{\mathbf r}\cap\wh\Lambda}
\le
\frac{6}{5\zeta(3)}\vol M_{\mathbf r}.
\]
\end{lemma}

Write \(\norm{\cdot}_\infty\) for the sup norm.  For
\(\mathbf r=(r_1,r_2,1)\), \(1\le r_1\le r_2\), and
\(0<s<1/2\), put
\[
D_{\mathbf r}=\diag(r_1,r_2,1)
\]
and define
\[
\mathcal S(\Lambda,\mathbf r,s)
=
\left\{
v\in M_{\mathbf r}\cap\wh\Lambda:
\begin{aligned}
&\exists\,\varphi\in\wh{\Lambda^*}
 \text{ with }\varphi\cdot v=0,\\[-1mm]
&\norm{(\varphi_1,\varphi_2)}_\infty\le s,\ 
 \norm{D_{\mathbf r}\varphi}_\infty\le3sr_2
\end{aligned}
\right\}.
\]

\begin{lemma}[{\cite[Lemma~3.9]{LSST}}]\label{lem:LSST-3.9}
Let \(\eps,t>0\), set
\[
s=\eps^2,\qquad
\mathbf r=(\eps e^t,\eps e^t,1),
\]
and let \(\Lambda\in\cL_3'\cap\cK_{\eps^2}^*(3)\).
There exists a positive constant, depending only on \(w\), such that
if \(\eps\) is smaller than this constant and
\[
e^{-w_2t/20}<\eps,
\]
then
\[
\card{
\mathcal S(a_t\Lambda,\mathbf r,s)
}
\le
\eps^{1/2}\vol M_{\mathbf r}.
\]
\end{lemma}

\begin{lemma}[{\cite[Lemma~3.10]{LSST}}]\label{lem:LSST-3.10}
Let \(\eps,t,s>0\), define
\[
\delta_w=\frac1{20}\min\{w_2,w_1-w_2\},
\qquad
\mathbf r=
(\eps e^{w_1t},\eps e^{(w_1+2w_2)t},1),
\]
\[
b_t=
\diag(e^{(w_1-w_2)t},e^{2w_2t},e^{-t}),
\]
and let \(\Lambda\in\cK_{\eps^2}^*(3)\).
There exists a positive constant, depending only on \(w\), such that
whenever \(s>0\) is smaller than this constant and
\[
e^{-\delta_wt}<\eps<s,
\]
one has
\[
\card{
\mathcal S(b_t\Lambda,\mathbf r,s)
}
\le
s\vol M_{\mathbf r}.
\]
\end{lemma}

\subsection{Construction of the variable-step tree}

Throughout this section, all constants are uniform across levels,
parents, and branches and may depend only on the fixed weight \(w\).
The auxiliary constants \(r\), \(\eps_*\), and \(d_*\) are fixed in
that order in \cref{thm:child-count} below.  Fix a base lattice
\(\Lambda^\circ\in\cL_3\).

Let
\[
0\le t_0<t_1<t_2<\cdots,
\qquad
d_n=t_n-t_{n-1}\quad(n\ge1),
\]
and let \((\eps_n)_{n\ge0}\) satisfy
\[
1\ge\eps_0\ge\eps_1\ge\cdots>0,
\qquad \eps_n\longrightarrow0.
\]
For \(n\ge0\) and \(\xi\in\R^2\), define the level-\(n\)
rectangle and lattice by
\begin{equation}\label{eq:node-rectangle}
\beta_n(\xi)
=
\prod_{i=1}^2
\left[
\xi_i-\eps_ne^{-w_it_{n+1}-t_n},
\xi_i+\eps_ne^{-w_it_{n+1}-t_n}
\right]
\end{equation}
and
\[
\Lambda_n(\xi)=a_{t_n}u_\xi\Lambda^\circ.
\]
For \(n\ge1\), set
\[
k_n=\diag(e^{-w_2d_n},e^{w_2d_n},1),
\]
and, for a parent center \(\kappa\), define the level-\(n\)
child-center region
\begin{equation}\label{eq:child-center-region}
\mathcal I_n(\kappa)
=
\prod_{i=1}^2
\left[
\kappa_i-\frac{\eps_n}{2}e^{-w_it_n-t_{n-1}},
\kappa_i+\frac{\eps_n}{2}e^{-w_it_n-t_{n-1}}
\right].
\end{equation}
Its children are
\begin{equation}
\mathcal C_n(\kappa)
=
\left\{
\theta\in\mathcal I_n(\kappa):
\begin{array}{l}
\Lambda_n(\theta)\in\cL_3',\\
\Lambda_n(\theta)\in\cK_{\eps_n^2}^*(3),\\
k_n\Lambda_n(\theta)\in\cK_r^*(3)
\end{array}
\right\}.
\end{equation}
The three conditions respectively record a primitive vertical return
vector and exclude the two dual degeneracies counted by
\cref{lem:LSST-3.9,lem:LSST-3.10}.

We use words as vertices rather than centers alone.  This makes the
parent map well defined even if two words have the same center.

Fix \(\kappa_0\in\R^2\) satisfying
\[
a_{t_0}u_{\kappa_0}\Lambda^\circ
\in\cL_3'\cap\cK_{\eps_0^2}^*(3).
\]
Let
\[
 \cT_0=\{\varnothing\},
 \qquad
 c(\varnothing)=\kappa_0,
 \qquad
 \beta(\varnothing)=\beta_0(\kappa_0).
\]
Recursively, for \(n\ge1\), each
\(\sigma\in\cT_{n-1}\) and each
\(\theta\in\mathcal C_n(c(\sigma))\) determine a child vertex
\(\omega=(\sigma,\theta)\in\cT_n\).  Equivalently,
\begin{equation}
 \cT_n
 =\left\{(\sigma,\theta):
 \sigma\in\cT_{n-1},\ 
 \theta\in\mathcal C_n(c(\sigma))
 \right\}.
\end{equation}
For \(\omega=(\sigma,\theta)\), set
\[
 c(\omega)=\theta,
 \qquad
 \beta(\omega)=\beta_n(c(\omega)),
\]
and declare \(\sigma\) to be the parent of \(\omega\).  An infinite
branch is a sequence \((\sigma_n)_{n\ge0}\), where
\(\sigma_n\in\cT_n\) and \(\sigma_{n-1}\) is the parent of
\(\sigma_n\) for \(n\ge1\).
Set \(\cT=\bigsqcup_{n\ge0}\cT_n\) and
\[
 F=\cF(\cT,\beta).
\]

\subsection{Counting children}

The child count is obtained by counting all primitive candidates and
then subtracting the two illegal classes.
\Cref{lem:candidate-shell,lem:vertical-transference,lem:box-minima}
give the total count, while \cref{lem:bad-candidate-injection} reduces
the illegal counts to \cref{lem:LSST-3.9,lem:LSST-3.10}.

For \(n\ge1\), let \(\cA_n(\kappa)\) denote the candidates
\(\theta\in\mathcal I_n(\kappa)\) satisfying only the
\(\cL_3'\) condition.  Thus
\[
 \mathcal C_n(\kappa)
 =\cA_n(\kappa)\setminus
 (\cB_{1,n}(\kappa)\cup\cB_{2,n}(\kappa)),
\]
where
\begin{align*}
 \cB_{1,n}(\kappa)
 &=\{\theta\in\cA_n(\kappa):
       \Lambda_n(\theta)\notin\cK_{\eps_n^2}^*(3)\},\\
 \cB_{2,n}(\kappa)
 &=\{\theta\in\cA_n(\kappa):
       k_n\Lambda_n(\theta)\notin\cK_r^*(3)\}.
\end{align*}

\begin{lemma}\label{lem:bad-candidate-injection}
Let \(n\ge1\), let \(\kappa\) be a parent center at time
\(t_{n-1}\), and put
\[
 d=d_n,\qquad \eps=\eps_n.
\]
Define
\[
 \Lambda_1=a_{t_n}u_\kappa\Lambda^\circ,
 \qquad
 \Lambda_2=k_na_{t_n}u_\kappa\Lambda^\circ,
\]
let
\[
 s_1=\eps^2,\qquad s_2=r,
\]
and set
\[
 \mathbf r_1=(\eps e^d,\eps e^d,1),
 \qquad
 \mathbf r_2=
 (\eps e^{w_1d},\eps e^{(w_1+2w_2)d},1).
\]
Write \(\mathbf r_i=(r_{i,1},r_{i,2},1)\) for \(i=1,2\).
Assume \(r_{i,2}\ge r_{i,1}\ge1\) and \(s_i<1/2\).  Then, for
\(i=1,2\), there is an injection
\begin{equation}\label{eq:bad-injection}
 \cB_{i,n}(\kappa)
 \hookrightarrow
 \mathcal S(\Lambda_i,\mathbf r_i,s_i).
\end{equation}
\end{lemma}

\begin{proof}
For a candidate \(\theta\), set
\[
 \Lambda_1(\theta)=a_{t_n}u_\theta\Lambda^\circ,
 \qquad
 \Lambda_2(\theta)=k_na_{t_n}u_\theta\Lambda^\circ.
\]
Define \(x^{(1)}(\theta),x^{(2)}(\theta)\in\R^2\) by
\[
\begin{aligned}
 x^{(1)}(\theta)
 &=\diag\!\left(e^{(1+w_1)t_n},e^{(1+w_2)t_n}\right)
   (\kappa-\theta),\\
 x^{(2)}(\theta)
 &=\diag\!\left(
   e^{(1+w_1)t_n-w_2d},
   e^{(1+w_2)t_n+w_2d}
   \right)(\kappa-\theta).
\end{aligned}
\]
A direct calculation gives
\[
\begin{aligned}
 u_{x^{(1)}(\theta)}a_{t_n}u_\theta
 &=a_{t_n}u_\kappa,\\
 u_{x^{(2)}(\theta)}k_na_{t_n}u_\theta
 &=k_na_{t_n}u_\kappa.
\end{aligned}
\]
and hence
\[
 u_{x^{(i)}(\theta)}\Lambda_i(\theta)=\Lambda_i
 \quad(i=1,2).
\]
Since \(\theta\) lies in the child-center region
\eqref{eq:child-center-region},
\[
 \norm{x^{(1)}(\theta)}_\infty
 \le\frac\eps2e^d\le\eps e^d,
\]
and
\[
 \abs{x_1^{(2)}(\theta)}
 \le\frac\eps2e^{w_1d}\le\eps e^{w_1d},
 \qquad
 \abs{x_2^{(2)}(\theta)}
 \le\frac\eps2e^{(w_1+2w_2)d}
 \le\eps e^{(w_1+2w_2)d}.
\]
Although the candidates lie in the child-center region, we retain the larger
radii \(\mathbf r_1,\mathbf r_2\) above.  The bounds place the associated
vectors in the boxes \(M_{\mathbf r_i}\), so the parameters in
\cref{lem:LSST-3.9,lem:LSST-3.10} require no rescaling.

Because \(\Lambda_1(\theta)\in\cL_3'\), put
\[
 \ell_\theta:=\ell_3\bigl(\Lambda_1(\theta)\bigr).
\]
Then \(1/2<\ell_\theta\le1\), and
\(v_\theta:=\ell_\theta e_3\) is the positive primitive generator of
\(\Lambda_1(\theta)\cap\R e_3\).  Since \(k_ne_3=e_3\), the vector
\(v_\theta\) is primitive in both \(\Lambda_1(\theta)\) and
\(\Lambda_2(\theta)\).  Put
\[
 z_i(\theta)=u_{x^{(i)}(\theta)}v_\theta.
\]
The lattice identity above gives \(z_i(\theta)\in\wh{\Lambda_i}\), and
\[
 z_i(\theta)
 =\ell_\theta
 (-x_1^{(i)}(\theta),-x_2^{(i)}(\theta),1).
\]
Consequently,
\[
 z_i(\theta)\in M_{\mathbf r_i}\cap\wh{\Lambda_i},
 \qquad
 z_{i,3}(\theta)>0.
\]
The map \(\theta\mapsto z_i(\theta)\) is injective.  Indeed,
\[
 \frac{(z_{i,1}(\theta),z_{i,2}(\theta))}{z_{i,3}(\theta)}
 =-x^{(i)}(\theta),
\]
and the diagonal matrix defining \(x^{(i)}(\theta)\) is invertible, so
these ratios determine \(\theta\).  Restricting to positive third coordinate removes the possible
\(\pm z_i\) multiplicity.

Now suppose \(\theta\in\cB_{i,n}(\kappa)\).  The failure of the
\(i\)-th dual-systole condition gives a primitive dual vector
\[
 \xi_i\in\Lambda_i(\theta)^*\setminus\{0\},
 \qquad
 \norm{\xi_i}<s_i.
\]
Since \(\Lambda_i(\theta)=u_{x^{(i)}(\theta)}^{-1}\Lambda_i\),
\[
 \Lambda_i(\theta)^*
 =u_{x^{(i)}(\theta)}^{\mathsf T}\Lambda_i^*.
\]
Thus there is a primitive \(\varphi_i\in\Lambda_i^*\) such that
\[
 \xi_i=u_{x^{(i)}(\theta)}^{\mathsf T}\varphi_i.
\]
Since \(z_i(\theta)\in\Lambda_i\) and \(\varphi_i\in\Lambda_i^*\),
\(\varphi_i\cdot z_i(\theta)\in\Z\).  On the other hand,
\[
 \abs{\varphi_i\cdot z_i(\theta)}
 =\abs{\xi_i\cdot v_\theta}
 \le\norm{\xi_i}
 <s_i<1.
\]
Therefore,
\[
 \varphi_i\cdot z_i(\theta)=0.
\]

Writing
\(\varphi_i=(\varphi_{i,1},\varphi_{i,2},\varphi_{i,3})\), the preceding
pullback identity gives
\[
 \xi_i
 =\bigl(
 \varphi_{i,1},
 \varphi_{i,2},
 \varphi_{i,3}
 -x_1^{(i)}(\theta)\varphi_{i,1}
 -x_2^{(i)}(\theta)\varphi_{i,2}
 \bigr).
\]
Thus
\[
 \norm{(\varphi_{i,1},\varphi_{i,2})}_\infty<s_i,
\]
and, using \(\abs{x_j^{(i)}(\theta)}\le r_{i,j}\),
\[
 \abs{\varphi_{i,3}}
 <s_i(1+r_{i,1}+r_{i,2})
 \le3s_ir_{i,2}.
\]
Together with \(r_{i,1}\le r_{i,2}\), this gives
\[
 \norm{D_{\mathbf r_i}\varphi_i}_\infty<3s_ir_{i,2}.
\]
Hence
\[
 z_i(\theta)\in\mathcal S(\Lambda_i,\mathbf r_i,s_i).
\]
Injectivity of \(\theta\mapsto z_i(\theta)\) proves
\eqref{eq:bad-injection}.
\end{proof}

\begin{lemma}\label{lem:candidate-shell}
Let \(n\ge1\), and let \(\kappa\) be a parent center at time
\(t_{n-1}\).  Put
\(d=d_n\), \(\eps=\eps_n\), and
\[
 \Lambda_1=a_d(a_{t_{n-1}}u_\kappa\Lambda^\circ)
 =a_{t_n}u_\kappa\Lambda^\circ.
\]
Set
\[
 M=\left\{z\in\R^3:
 |z_j|\le\frac\eps2 e^d|z_3|\ (j=1,2),
 \quad \frac12<|z_3|\le1\right\}.
\]
For each \(\theta\in\cA_n(\kappa)\), put
\[
 \ell_\theta:=\ell_3\bigl(a_{t_n}u_\theta\Lambda^\circ\bigr).
\]
Then \(\ell_\theta e_3\) is the positive primitive vertical vector in
\(a_{t_n}u_\theta\Lambda^\circ\), and the map
\begin{equation}\label{eq:candidate-shell-forward}
 \theta\longmapsto
 z(\theta)=a_{t_n}u_{\kappa-\theta}a_{t_n}^{-1}
 (\ell_\theta e_3)
\end{equation}
is a bijection from \(\cA_n(\kappa)\) onto
\[
 \{z\in M\cap\wh\Lambda_1:z_3>0\}.
\]
Therefore,
\begin{equation}
 \card{\cA_n(\kappa)}
 =\frac12\card{M\cap\wh\Lambda_1}.
\end{equation}
\end{lemma}

\begin{proof}
For a candidate \(\theta\), the conjugation identity gives
\[
 z(\theta)=\ell_\theta
 \left(e^{(1+w_1)t_n}(\theta_1-\kappa_1),
       e^{(1+w_2)t_n}(\theta_2-\kappa_2),1\right).
\]
The bounds defining \(\mathcal I_n(\kappa)\) imply \(z(\theta)\in M\), and
conjugation carries primitive vectors to primitive vectors, so
\(z(\theta)\in\wh\Lambda_1\).  Its third coordinate is positive.

Conversely, write
\[
 z=a_{t_n}u_\kappa v
 \in M\cap\wh\Lambda_1,
 \qquad z_3>0,
 \qquad
 v=(p_1,p_2,q)\in\wh{\Lambda^\circ}.
\]
Then \(q=e^{t_n}z_3>0\) and
\(e^{t_n}/2<q\le e^{t_n}\).  Define
\[
 \theta=\frac{(p_1,p_2)}q.
\]
Then
\[
 a_{t_n}u_\theta v=qe^{-t_n}e_3,
 \qquad \frac12<qe^{-t_n}\le1.
\]
Since \(v\) is primitive in \(\Lambda^\circ\), it generates
\(\Lambda^\circ\cap u_\theta^{-1}(\R e_3)\).  Hence
\[
 a_{t_n}u_\theta\Lambda^\circ\cap\R e_3
 =qe^{-t_n}\Z e_3,
\]
so \(a_{t_n}u_\theta\Lambda^\circ\in\cL_3'\).  Moreover,
\[
 \frac{z_j}{z_3}=e^{(1+w_j)t_n}(\theta_j-\kappa_j),
\]
and the shell inequalities are exactly the inequalities defining
\(\mathcal I_n(\kappa)\) in \eqref{eq:child-center-region}.  This constructs the inverse of
\eqref{eq:candidate-shell-forward}.  Finally, \(M\) is centrally
symmetric and contains no vector with \(z_3=0\), so exactly one member
of each pair \(\{z,-z\}\) has positive third coordinate.  This proves
the factor \(1/2\).
\end{proof}

\begin{lemma}\label{lem:vertical-transference}
Let \(0<\rho\le1\), and let
\[
 \Lambda\in\cL_3'\cap\cK_\rho^*(3).
\]
Then
\[
 \lambda_1(B_1,\Lambda)
 \ge \min\{\rho,1/2\}
 \ge \frac\rho2.
\]
\end{lemma}

\begin{proof}
Let \(\pi(v_1,v_2,v_3)=(v_1,v_2)\).
If \(v\in(\Lambda\cap\R e_3)\setminus\{0\}\), then
\[
 \norm v\ge\ell_3(\Lambda)>\frac12.
\]
If \(v\notin\R e_3\), then
\[
 v\times \ell_3(\Lambda)e_3\in\Lambda^*\setminus\{0\},
\]
because
\[
 \bigl(v\times \ell_3(\Lambda)e_3\bigr)\cdot z
 =\det\bigl(v,\ell_3(\Lambda)e_3,z\bigr)\in\Z
\]
for every \(z\in\Lambda\).  Hence
\[
 \rho
 \le \norm{v\times \ell_3(\Lambda)e_3}
 =\ell_3(\Lambda)\norm{\pi(v)}
 \le\norm v.
\]
Taking the minimum over nonzero \(v\in\Lambda\) proves the claim.
\end{proof}

\begin{lemma}\label{lem:box-minima}
Let \(d>0\) and \(0<\eps\le1\).  Assume
\(\Lambda\in\cL_3'\cap\cK_{\eps^2}^*(3)\), and put
\(\Lambda_1=a_d\Lambda\).  Define
\[
 Q^{(1)}=
 \left\{z:\max(|z_1|,|z_2|)\le\frac14\eps e^d,
 |z_3|\le1\right\},
\]
\[
 Q^{(2)}=
 \left\{z:\max(|z_1|,|z_2|)\le\frac14\eps e^d,
 |z_3|\le\frac12\right\}.
\]
For each
\(Q\in\{Q^{(1)},Q^{(2)},2Q^{(2)}\}\),
\begin{equation}\label{eq:box-minima}
 \lambda_1(Q,\Lambda_1)\ge\frac{1}{2\sqrt3}e^{-d}\eps^2,
 \qquad
 \lambda_3(Q,\Lambda_1)\le24e^{-w_2d}\eps^{-3}.
\end{equation}
\end{lemma}

\begin{proof}
By \cref{lem:vertical-transference} and Mahler's transference
inequality in dimension three \cite[Chapter~VIII]{Cassels},
\[
 \lambda_1(B_1,\Lambda)\ge\frac12\eps^2,
 \qquad
 \lambda_3(B_1,\Lambda)
 \le\frac6{\lambda_1(B_1,\Lambda^*)}
 \le6\eps^{-2}.
\]
For all three boxes,
\[
 a_d^{-1}Q\subseteq B_{\sqrt3e^d}.
\]
Therefore a nonzero vector of \(\Lambda\) in
 \(\lambda a_d^{-1}Q\), for \(\lambda>0\), has Euclidean norm at most
\(\sqrt3\lambda e^d\), which gives
\[
 \lambda_1(Q,\Lambda_1)
 =\lambda_1(a_d^{-1}Q,\Lambda)
 \ge\frac{1}{2\sqrt3}e^{-d}\eps^2.
\]
On the other hand, direct comparison of the three semi-axes gives
\[
 B_{\frac14\eps e^{w_2d}}\subseteq a_d^{-1}Q.
\]
Therefore
\[
 \lambda_3(Q,\Lambda_1)
 =\lambda_3(a_d^{-1}Q,\Lambda)
 \le4(\eps e^{w_2d})^{-1}\lambda_3(B_1,\Lambda)
 \le24e^{-w_2d}\eps^{-3}.
\]
\end{proof}

\begin{theorem}\label{thm:child-count}
There exist positive constants \(r,\eps_*,d_*\), chosen in the order
\begin{equation}\label{eq:parameter-order}
 r\longrightarrow\eps_*\longrightarrow d_*,
\end{equation}
and constants \(0<c_-<c_+<\infty\) such that the following holds.
Let \(n\ge1\), and let \(\kappa\) be a parent center at time
\(t_{n-1}\).  Put \(d=d_n\) and \(\eps=\eps_n\), and suppose that,
for some parent scale \(\eta\in(0,1]\),
\[
\Lambda:=a_{t_{n-1}}u_\kappa\Lambda^\circ
\in\cL_3'\cap\cK_{\eta^2}^*(3).
\]
Assume
\[
0<\eps\le\min\{\eta,\eps_*\},
\qquad d\ge d_*,
\]
and
\begin{equation}\label{eq:count-condition}
 \left(\frac2\eps\right)^{100}
 \le
 \min\{e^{w_2d},e^{(w_1-w_2)d}\}.
\end{equation}
Then
\begin{equation}
 c_-\eps^2e^{2d}
 \le
 \card{\mathcal C_n(\kappa)}
 \le
 c_+\eps^2e^{2d}.
\end{equation}
\end{theorem}

\begin{proof}
Let
\[
 \Lambda=a_{t_{n-1}}u_\kappa\Lambda^\circ,
 \qquad
 \Lambda_1=a_d\Lambda,
 \qquad
 V=\eps^2e^{2d}.
\]
The parent hypothesis and \(\eps\le\eta\) imply
\[
 \Lambda\in\cL_3'\cap\cK_{\eps^2}^*(3).
\]
By \cref{lem:candidate-shell},
\begin{equation*}
 \card{\cA_n(\kappa)}
 =\frac12\card{M\cap\wh{\Lambda_1}},
\end{equation*}
where \(M\) is the shell defined there.  Let \(Q^{(1)}\)
and \(Q^{(2)}\) be the boxes from \cref{lem:box-minima}.  Then
\begin{equation}\label{eq:box-sandwich}
 Q^{(1)}\setminus Q^{(2)}
 \subseteq M
 \subseteq2Q^{(2)},
\end{equation}
and
\[
 \vol Q^{(1)}=\frac12V,
 \qquad
 \vol Q^{(2)}=\frac14V,
 \qquad
 \vol(2Q^{(2)})=2V.
\]
For each of the three boxes in \eqref{eq:box-sandwich}, the
successive-minima bounds are exactly those in \eqref{eq:box-minima}.
Put \(u=|\log\eps|\) and
\(\gamma_w=\min\{w_2,w_1-w_2\}\).  Condition
\eqref{eq:count-condition} implies
\[
 100(u+\log2)\le\gamma_wd.
\]
Therefore, uniformly for the three boxes,
\[
 \lambda_3(Q,\Lambda_1)
 \le 24
 \exp\left[-\left(w_2-\frac{3\gamma_w}{100}\right)d\right],
\]
and
\[
 -\lambda_3(Q,\Lambda_1)\log\lambda_1(Q,\Lambda_1)
 \le K_0(1+d)
 \exp\left[-\left(w_2-\frac{3\gamma_w}{100}\right)d\right].
\]
Here \(K_0>0\) depends only on \(w\).
Since \(w_2-3\gamma_w/100>0\), we may choose \(d_*\) so that both
quantities are at most the constant in \cref{lem:LSST-3.7} whenever
\(d\ge d_*\).

The lattice \(\Lambda_1\) is unimodular, so
\cref{lem:LSST-3.7} and \eqref{eq:box-sandwich} give
fixed constants \(0<A_-<A_+<\infty\) such that
\begin{equation}\label{eq:total-candidates}
 A_-V
 \le
 \card{\cA_n(\kappa)}
 \le
 A_+V.
\end{equation}
For example, before the positive-sign factor is applied, the lower
main-term difference is
\[
 \frac{4}{5\zeta(3)}\,\frac V2
 -\frac{6}{5\zeta(3)}\,\frac V4
 =\frac{V}{10\zeta(3)}>0.
\]

It remains to estimate the two classes of bad candidates.
By \cref{lem:bad-candidate-injection,lem:LSST-3.9,lem:LSST-3.10},
\begin{align}
 \card{\cB_{1,n}(\kappa)}
 &\le 8\sqrt\eps\,V,\\
 \card{\cB_{2,n}(\kappa)}
 &\le 8r\,V.
\end{align}
Indeed,
\[
 \vol M_{\mathbf r_1}=8\eps^2e^{2d}=8V,
 \qquad
 \vol M_{\mathbf r_2}=8\eps^2e^{2d}=8V,
\]
and
\[
 \Lambda_1=a_d\Lambda,
 \qquad
 \Lambda_2=k_na_d\Lambda
 =\diag(e^{(w_1-w_2)d},e^{2w_2d},e^{-d})\Lambda,
\]
are exactly the lattices to which
\cref{lem:LSST-3.9,lem:LSST-3.10} are applied.

The parameters are now fixed in the order
\eqref{eq:parameter-order}.  First choose
\(0<r<\min\{1/2,A_-/32\}\) below the fixed \(s\)-threshold in
\cref{lem:LSST-3.10}.  Then choose
\(0<\eps_*<\min\{r,2^{-1/2}\}\) below the fixed
\(\eps\)-threshold in \cref{lem:LSST-3.9}, and so that
\[
 8\sqrt{\eps_*}\le\frac14A_-.
\]
Finally, enlarge \(d_*\) so that the hypotheses of the primitive
lattice-point estimate, the inequalities
\(r_{i,2}\ge r_{i,1}\ge1\), and the two exponential conditions in
\cref{lem:LSST-3.9,lem:LSST-3.10} all follow from
\eqref{eq:count-condition}.
Then
\[
 \card{\mathcal C_n(\kappa)}
 \ge A_-V-8\sqrt\eps V-8rV
 \ge\frac12A_-V.
\]
The upper bound follows from \eqref{eq:total-candidates}.  Thus one may
take \(c_-=A_-/2\) and \(c_+=A_+\).
\end{proof}

\subsection{Regularity, separation, and singularity}

The child count gives branching but not yet a usable self-affine
structure.  We now verify nesting and shrinking, obtain quantitative
sibling separation, and finally prove that every branch limit gives a
divergent trajectory.

Set
\begin{equation}\label{eq:SL}
 S_n=2\eps_ne^{-w_1t_{n+1}-t_n},
 \qquad
 L_n=2\eps_ne^{-w_2t_{n+1}-t_n}.
\end{equation}
Thus every level \(n\) rectangle has short side \(S_n\) and long
side \(L_n\), and \(S_n\le L_n\).

\begin{proposition}\label{prop:regularity}
Assume that, for every \(n\ge1\),
\begin{enumerate}[label=\textup{(A\arabic*)}]
\item the nesting condition
\begin{equation}\label{eq:regularity-nesting}
 \eps_n\left(\frac12+e^{-d_n-w_2d_{n+1}}\right)
 \le\eps_{n-1}
\end{equation}
holds;
\item the hypotheses of \cref{thm:child-count} hold and
\begin{equation}\label{eq:nonempty-count}
 c_-\eps_n^2e^{2d_n}\ge1;
\end{equation}
\item \(t_n\to\infty\).
\end{enumerate}
Then the construction above defines a regular self-affine structure
\((\cT,\beta)\) in the sense of LSST\@.  More explicitly:
\begin{enumerate}[label=\textup{(\roman*)}]
\item every vertex has a nonempty finite set of children;
\item the rectangle of every nonroot vertex is contained in the
rectangle of its parent;
\item for every branch \((\sigma_n)\),
\(\diam\beta(\sigma_n)\to0\);
\item all level \(n\) rectangles have the common side lengths
\(S_n,L_n\) in \eqref{eq:SL}.
\end{enumerate}
Therefore, every branch determines a unique point of \(F\).
\end{proposition}

\begin{proof}
By \cref{thm:child-count} and
\eqref{eq:nonempty-count}, every vertex has a finite, nonempty set of
children.  In coordinate \(i\), the displacement of a child center
from its parent center is at most
\(\frac12\eps_ne^{-w_it_n-t_{n-1}}\), while the child half-length is
\[
 \eps_ne^{-w_it_{n+1}-t_n}
 =\eps_ne^{-w_it_n-t_{n-1}}e^{-d_n-w_id_{n+1}}.
\]
The worst coordinate is \(i=2\), so
\eqref{eq:regularity-nesting} places every child rectangle inside its
parent rectangle.

Finally, for every level-\(n\) center \(\theta\),
\[
 \diam\beta_n(\theta)
 \le\sqrt{S_n^2+L_n^2}
 \le\sqrt2L_n
 \le2\sqrt2e^{-(1+w_2)t_n},
\]
because \(\eps_n\le1\).  This tends to zero.  The formulas for the common side lengths, together with
\(S_n\le L_n\), follow from \eqref{eq:node-rectangle} and
\eqref{eq:weight}.
\end{proof}

\begin{lemma}
\label{lem:separation}
Assume that \(n\ge1\) and
\begin{equation}\label{eq:separation-conditions}
 \eps_n\le\frac r8,
 \qquad
 \eps_ne^{w_2(d_n-d_{n+1})}\le\frac r8.
\end{equation}
If \(\theta\ne\theta'\) are two distinct level-\(n\) children of the same
level-\(n-1\) parent, define
\begin{equation}\label{eq:rho-formula}
 \rho_n
 =
 \min\left\{
 1,\,
 \frac{r}{8\eps_{n-1}}
 \min\left(
 e^{-w_1d_n},
 e^{(w_1-w_2)t_n-(1+w_2)d_n}
 \right)
 \right\}.
\end{equation}
Then
\begin{equation}\label{eq:sibling-separation}
 \dist\bigl(\beta_n(\theta),\beta_n(\theta')\bigr)
 \ge S_{n-1}\rho_n.
\end{equation}
Therefore, in the indexing of \cref{thm:LSST-2.1}, distinct
children of a level \(n\) vertex are separated by
\(\rho_{n+1}S_n\).
\end{lemma}

\begin{proof}
Put
\[
 G_n=k_na_{t_n}
 =
 \diag\left(
 e^{w_1t_n-w_2d_n},
 e^{w_2t_n+w_2d_n},
 e^{-t_n}
 \right).
\]
Since \(k_ne_3=e_3\) and both children satisfy the first legality
condition, put
\[
 \ell_\theta:=\ell_3\bigl(\Lambda_n(\theta)\bigr),
 \qquad
 \ell_{\theta'}:=\ell_3\bigl(\Lambda_n(\theta')\bigr).
\]
Then \(1/2<\ell_\theta,\ell_{\theta'}\le1\), and
\[
 \ell_\theta e_3\in G_nu_\theta\Lambda^\circ,
 \qquad
 \ell_{\theta'}e_3\in G_nu_{\theta'}\Lambda^\circ.
\]
Write
\[
 y=\theta'-\theta
\]
and transport the first vertical vector to the second child lattice:
\begin{equation}\label{eq:transported-vertical-vector}
 v
 =
 G_nu_y G_n^{-1}(\ell_\theta e_3).
\end{equation}
Indeed, if \(\ell_\theta e_3=G_nu_\theta m\) with
\(m\in\Lambda^\circ\), then
\[
 v
 =
 G_nu_y u_\theta m
 =
 G_nu_{\theta'}m
 \in G_nu_{\theta'}\Lambda^\circ.
\]
Thus \(v\) and \(\ell_{\theta'}e_3\) belong to the unimodular lattice
\[
 \Gamma'=G_nu_{\theta'}\Lambda^\circ=k_n\Lambda_n(\theta').
\]
Since \(\theta\ne\theta'\), these two vectors are linearly independent.
Thus \(v\times \ell_{\theta'}e_3\) is a nonzero element of \((\Gamma')^*\),
because, for every \(z\in\Gamma'\),
\[
 (v\times \ell_{\theta'}e_3)\cdot z
 =\det(v,\ell_{\theta'}e_3,z)\in\Z,
\]
where the last inclusion follows from the unimodularity of
\(\Gamma'\).  The third legality condition
\(\Gamma'\in\cK_r^*(3)\) therefore gives
\begin{equation}
 \norm{v\times \ell_{\theta'}e_3}\ge r.
\end{equation}

Diagonal conjugation in
\eqref{eq:transported-vertical-vector} gives
\[
 v
 =
 \ell_\theta\left(
 -e^{(1+w_1)t_n-w_2d_n}y_1,\,
 -e^{(1+w_2)t_n+w_2d_n}y_2,\,
 1
 \right).
\]
Hence
\begin{align}
 r
 &\le
 \ell_\theta\ell_{\theta'}
 \left[
 e^{2((1+w_1)t_n-w_2d_n)}|y_1|^2
 +
 e^{2((1+w_2)t_n+w_2d_n)}|y_2|^2
 \right]^{1/2}\notag\\
 &\le
 \left[
 e^{2((1+w_1)t_n-w_2d_n)}|y_1|^2
 +
 e^{2((1+w_2)t_n+w_2d_n)}|y_2|^2
 \right]^{1/2}.
 \label{eq:two-coordinate-bound}
\end{align}
It follows that at least one of the following alternatives holds:
\begin{align}
 |y_1|
 &\ge
 \frac r2
 e^{-(1+w_1)t_n+w_2d_n},
 \label{eq:horizontal-center-separation}\\
 |y_2|
 &\ge
 \frac r2
 e^{-(1+w_2)t_n-w_2d_n}.
 \label{eq:vertical-center-separation}
\end{align}
Indeed, if both scaled coordinates were smaller than \(r/2\),
then the right-hand side of \eqref{eq:two-coordinate-bound} would be
smaller than \(r/\sqrt2<r\), a contradiction.

Let
\[
 h_{n,i}=\eps_ne^{-w_it_{n+1}-t_n}
 \qquad(i=1,2)
\]
be the coordinate half-lengths of a level \(n\) rectangle.

Suppose first that \eqref{eq:horizontal-center-separation} holds.
Then
\begin{align*}
 2h_{n,1}
 &=
 2\eps_ne^{-(1+w_1)t_n-w_1d_{n+1}}\\
 &=
 \frac r2e^{-(1+w_1)t_n+w_2d_n}
 \left(
 \frac{4\eps_n}{r}
 e^{-w_1d_{n+1}-w_2d_n}
 \right)\\
 &\le
 \frac12\,
 \frac r2e^{-(1+w_1)t_n+w_2d_n},
\end{align*}
where the last inequality uses \(\eps_n\le r/8\).
Therefore the gap between the two rectangles in the first coordinate
is at least
\begin{equation}\label{eq:first-rectangle-gap}
 \frac r4e^{-(1+w_1)t_n+w_2d_n}.
\end{equation}

If instead \eqref{eq:vertical-center-separation} holds, then
\begin{align*}
 2h_{n,2}
 &=
 2\eps_ne^{-(1+w_2)t_n-w_2d_{n+1}}\\
 &=
 \frac r2e^{-(1+w_2)t_n-w_2d_n}
 \left(
 \frac{4\eps_n}{r}
 e^{w_2(d_n-d_{n+1})}
 \right)\\
 &\le
 \frac12\,
 \frac r2e^{-(1+w_2)t_n-w_2d_n}
\end{align*}
by the second condition in
\eqref{eq:separation-conditions}.  Thus the second-coordinate
gap is at least
\begin{equation}\label{eq:second-rectangle-gap}
 \frac r4e^{-(1+w_2)t_n-w_2d_n}.
\end{equation}

Finally,
\[
 S_{n-1}
 =
 2\eps_{n-1}e^{-w_1t_n-t_{n-1}},
 \qquad
 t_{n-1}=t_n-d_n.
\]
Dividing \eqref{eq:first-rectangle-gap} by \(S_{n-1}\) gives
\[
 \frac{r}{8\eps_{n-1}}e^{-w_1d_n},
\]
while dividing \eqref{eq:second-rectangle-gap} by \(S_{n-1}\)
gives
\[
 \frac{r}{8\eps_{n-1}}
 e^{(w_1-w_2)t_n-(1+w_2)d_n}.
\]
Whenever one of the coordinate gaps is positive, the Euclidean
distance between the two axis-parallel rectangles is at least that
gap.  Taking the smaller of the two possible lower bounds and then
truncating at \(1\) proves
\eqref{eq:sibling-separation}--\eqref{eq:rho-formula}.
\end{proof}

\begin{lemma}\label{lem:singularity-shift}
Assume that \((\eps_n)\) is nonincreasing and
\begin{equation}
 \eps_n\to0,
 \qquad
 d_n\to\infty,
 \qquad
 g_n:=|\log\eps_n|=o(d_n).
\end{equation}
Then every branch limit point \(x\in F\) satisfies
\[
 \min_{v\in a_tu_x\Lambda^\circ\setminus\{0\}}\norm{v}
 \longrightarrow0.
\]
\end{lemma}

\begin{proof}
The crossover argument requires the shifted estimate
\begin{equation}\label{eq:shift-needed}
 g_{n-1}=o(d_{n+1}).
\end{equation}
Monotonicity of \(g_n\) gives
\[
 0\le\frac{g_{n-1}}{d_{n+1}}
 \le\frac{g_{n+1}}{d_{n+1}}
 \longrightarrow0,
\]
where the last limit is \(g_k=o(d_k)\) with \(k=n+1\).
Thus \eqref{eq:shift-needed} follows from the stated hypotheses.

Let \((\sigma_n)_{n\ge0}\) be an infinite branch, put
\(\kappa_n=c(\sigma_n)\), and let
\(x\in\bigcap_n\beta(\sigma_n)\).  Since
\(\Lambda_n(\kappa_n)=a_{t_n}u_{\kappa_n}\Lambda^\circ\in\cL_3'\),
its positive primitive vertical vector is
\(v_n:=\ell_3\bigl(\Lambda_n(\kappa_n)\bigr)e_3\).
Choose the corresponding primitive vector \(z_n\in\Lambda^\circ\), so that
\[
a_{t_n}u_{\kappa_n}z_n=v_n.
\]
For \(t\in[t_n,t_{n+1}]\), the vector \(a_tu_xz_n\) has Euclidean norm
\begin{equation}\label{eq:current-vector-bound}
 \ll
 \max\left\{
 \eps_ne^{-w_1(t_{n+1}-t)},
 \eps_ne^{-w_2(t_{n+1}-t)},
 e^{-(t-t_n)}
 \right\}.
\end{equation}
Define the crossover time
\begin{equation}
 t_n'=t_n+\frac{g_{n-1}}{1+w_1}.
\end{equation}
By \eqref{eq:shift-needed}, \(t_n'\le t_{n+1}\) for all large
\(n\).  On \([t_n,t_n']\), use \(a_tu_xz_{n-1}\).  Its norm is
\[
 \ll\max\left\{
 e^{-g_{n-1}/(1+w_1)},e^{-d_n}
 \right\}.
\]
On \([t_n',t_{n+1}]\), use \(a_tu_xz_n\) in
\eqref{eq:current-vector-bound}.  Its norm is
\[
 \ll\max\left\{
 \eps_n,e^{-g_{n-1}/(1+w_1)}
 \right\}.
\]
Both bounds tend to zero because \(\eps_n\to0\),
\(g_{n-1}\to\infty\), and \(d_n\to\infty\).  Mahler's compactness
criterion \cite[Chapter V]{Cassels} then shows that
\(a_tu_x\Lambda^\circ\) diverges, equivalently
\[
 \min_{v\in a_tu_x\Lambda^\circ\setminus\{0\}}\norm{v}\to0.
\]
\end{proof}

\subsection{Hausdorff dimension of the limit set}

The lower bound follows by inserting cumulative branching, side
lengths, and sibling separation into \cref{thm:LSST-2.1}; the remaining
point is to control the crossover between the long and short sides.
For \(\Lambda^\circ=\Z^3\), \cref{lem:singularity-shift} and
\eqref{eq:LSST-dimension} give the reverse inequality.

\begin{proposition}\label{prop:tree-dimension}
Assume that \((\cT,\beta)\) is regular and satisfies the separation
estimate \eqref{eq:sibling-separation}.  Suppose
\begin{equation}
\begin{aligned}
&0<\eps_{n+1}\le\eps_n\le1,
\qquad \eps_n\to0,
\qquad d_n\to\infty,\\
&d_{n+1}=o(t_n),
\qquad g_n=|\log\eps_n|=o(d_n).
\end{aligned}
\end{equation}
and suppose every level \(n-1\) vertex has at least
\begin{equation}\label{eq:Cn}
 C_n=c_-\eps_n^2e^{2d_n}
\end{equation}
children.
Then
\[
 \dim_H F\ge s_w.
\]
If, in addition, \(\Lambda^\circ=\Z^3\), then
\(\dim_H F=s_w\).
\end{proposition}

\begin{proof}
We apply the self-affine dimension estimate in
\cref{thm:LSST-2.1}.
Set \(C_0=1\), \(\rho_0=1\), and
\[
 P_n=\prod_{k=0}^nC_k.
\]
We begin with the cumulative error estimates.  Shifting the hypothesis
\(d_{n+1}=o(t_n)\) by one index gives \(d_n=o(t_{n-1})\).  Since
\(t_{n-1}\le t_n\), it follows that
\begin{equation}\label{eq:dn-small}
 d_n=o(t_n).
\end{equation}
Since \(g_k=o(d_k)\), for every \(\delta>0\) there is \(N_\delta\) such
that \(g_k\le\delta d_k\) for \(k\ge N_\delta\).  Hence
\begin{equation}
 \sum_{k=1}^ng_k=o(t_n).
\end{equation}
Also \(d_k\to\infty\) implies
\begin{equation}
 n=o(t_n).
\end{equation}
Indeed, for any \(M>0\), choose \(N_M\) so that \(d_k\ge M\) for
every \(k\ge N_M\).  Then
\(M(n-N_M)\le t_n-t_{N_M}\).
Using \eqref{eq:Cn},
\begin{align}
 \log P_n
 &=\sum_{k=1}^n
 \bigl(2d_k-2g_k+\log c_-\bigr)\notag\\
 &=2(t_n-t_0)-o(t_n)\notag\\
 &=2t_n-o(t_n),
 \label{eq:log-Pn}
\end{align}

The two side lengths satisfy the exact formulas
\begin{align}
 -\log S_n
 &=(1+w_1)t_n+g_n+w_1d_{n+1}-\log2,
 \label{eq:logS-exact}\\
 -\log L_n
 &=(1+w_2)t_n+g_n+w_2d_{n+1}-\log2.
 \label{eq:logL-exact}
\end{align}
The hypotheses imply \(g_n=o(t_n)\), so
\begin{equation}\label{eq:SL-asymptotics}
 \log S_n=-(1+w_1)t_n+o(t_n),
 \qquad
 \log L_n=-(1+w_2)t_n+o(t_n).
\end{equation}
Because \(\eps_n\) is nonincreasing,
\[
 \frac{L_{k+1}}{L_k}
 =\frac{\eps_{k+1}}{\eps_k}
 e^{-d_{k+1}-w_2d_{k+2}}<1.
\]
Thus \(L_k\) is decreasing.  Define the crossover index
\begin{equation}
 D_n=\max\{k\ge n:L_k\ge S_n\}.
\end{equation}
It is finite, and
\begin{equation}\label{eq:crossover-inequalities}
 L_{D_n}\ge S_n>L_{D_n+1}.
\end{equation}

Using the exact formulas \eqref{eq:logS-exact}--\eqref{eq:logL-exact},
\eqref{eq:crossover-inequalities} is equivalent to
\begin{equation}\label{eq:crossover-errors}
\begin{aligned}
 &(1+w_2)t_{D_n}+g_{D_n}+w_2d_{D_n+1}\\
 &\qquad\le (1+w_1)t_n+g_n+w_1d_{n+1}\\
 &\qquad<(1+w_2)t_{D_n+1}+g_{D_n+1}+w_2d_{D_n+2}.
\end{aligned}
\end{equation}
The first inequality and \(g_n,d_{n+1}=o(t_n)\) give
\(t_{D_n}=O(t_n)\).  Since \(D_n\ge n\to\infty\), the relevant
asymptotic estimates remain valid along the index sequences \(D_n\)
and \(D_n+1\).  In particular,
\[
 g_{D_n}=o(t_n),
 \qquad
 d_{D_n+1}=o(t_n).
\]
Thus \(t_{D_n+1}=O(t_n)\), and similarly
\[
 g_{D_n+1}=o(t_n),
 \qquad
 d_{D_n+2}=o(t_n).
\]
The two sides of \eqref{eq:crossover-errors} now yield
\begin{equation}\label{eq:crossover-ratio}
 \frac{t_{D_n}}{t_n}
 \longrightarrow
 \frac{1+w_1}{1+w_2}.
\end{equation}

Next,
\begin{equation}\label{eq:Dn-minus-n}
 D_n-n=o(t_n).
\end{equation}
Indeed, \eqref{eq:crossover-ratio} gives
\(t_{D_n}-t_n=O(t_n)\).  For any \(M>0\), all sufficiently large
\(d_i\) are at least \(M\), so
\[
 M(D_n-n)
 \le\sum_{i=n+1}^{D_n}d_i
 =t_{D_n}-t_n=O(t_n).
\]
Letting \(M\to\infty\) proves \eqref{eq:Dn-minus-n}.  The same
tail comparison with \(g_i=o(d_i)\) gives
\begin{equation}\label{eq:cross-error-sums}
 \sum_{i=n+1}^{D_n}g_i=o(t_n).
\end{equation}

We now estimate the separation factors.  The quotient of the second
inner term in \eqref{eq:rho-formula} by the first is
\[
 \exp\bigl((w_1-w_2)t_n-2w_2d_n\bigr)\longrightarrow\infty,
\]
by \eqref{eq:dn-small}.  Therefore the first inner term is the
minimum for all large \(n\).  Since \(\eps_{n-1}\le1\),
\begin{equation}\label{eq:rho-lower}
 \rho_n
 \ge\frac r8e^{-w_1d_n},
 \qquad
 \log\rho_n\ge-w_1d_n+\log(r/8).
\end{equation}
Since also \(0<\rho_{n+1}\le1\), we have the two-sided estimate
\[
 -w_1d_{n+1}+\log(r/8)
 \le \log\rho_{n+1}\le0.
\]
Together with \(d_{n+1}=o(t_n)\), this gives
\begin{equation}\label{eq:rho-next-small}
 \log\rho_{n+1}=o(t_n).
\end{equation}
Combining \eqref{eq:Cn}, \eqref{eq:rho-lower},
\eqref{eq:Dn-minus-n}, and \eqref{eq:cross-error-sums}, we obtain
\begin{align}
 \sum_{i=n+1}^{D_n}\log(\rho_iC_i)
 &\ge
 \sum_{i=n+1}^{D_n}
 \bigl((2-w_1)d_i-2g_i+\log(c_-r/8)\bigr)\notag\\
 &=(2-w_1)(t_{D_n}-t_n)-o(t_n).
 \label{eq:cross-product}
\end{align}
This is the cross-scale contribution appearing in
\cref{thm:LSST-2.1}.

Fix \(1<s<s_w\).  Using \eqref{eq:log-Pn},
\eqref{eq:SL-asymptotics}, \eqref{eq:crossover-ratio},
\eqref{eq:rho-next-small}, and \eqref{eq:cross-product}, the
logarithm in the numerator of \cref{thm:LSST-2.1} satisfies
\begin{align*}
&\log P_n+s\log S_n+s\log\rho_{n+1}
 +\sum_{i=n+1}^{D_n}\log(\rho_iC_i)\\
&\qquad\ge
 \left[
 2-s(1+w_1)
 +(2-w_1)\left(
 \frac{1+w_1}{1+w_2}-1
 \right)
 +o(1)
 \right]t_n\\
&\qquad=\left[(1+w_1)(s_w-s)+o(1)\right]t_n.
\end{align*}
Here we used \(2-w_1=1+w_2\), \(w_1+w_2=1\), and
\((1+w_1)s_w=1+2w_1\).  Since \(s<s_w\), the last lower bound
is a positive multiple of \(t_n\) for large \(n\).  Finally,
\eqref{eq:Dn-minus-n} gives \(\max\{D_n-n,1\}=o(t_n)\), and hence
\[
 \frac{
 \log P_n+s\log S_n+s\log\rho_{n+1}
 +\sum_{i=n+1}^{D_n}\log(\rho_iC_i)
 }{\max\{D_n-n,1\}}
 \longrightarrow+\infty.
\]
By \cref{thm:LSST-2.1}, \(\dim_H F\ge s\).  Letting
\(s\uparrow s_w\) proves the lower bound.  If
\(\Lambda^\circ=\Z^3\), then the singularity lemma gives
\(F\subseteq\Sing_w(2)\), and \eqref{eq:LSST-dimension} gives the
reverse bound.
\end{proof}

Thus every schedule satisfying the hypotheses of
\cref{lem:singularity-shift,prop:tree-dimension} produces a limit set
of Hausdorff dimension at least \(s_w\) along which
\(a_tu_x\Lambda^\circ\) diverges.  When
\(\Lambda^\circ=\Z^3\), the limit set is contained in
\(\Sing_w(2)\) and has Hausdorff dimension exactly \(s_w\).  The next
section selects such a schedule while enforcing the prescribed lower
envelope at all intermediate times.

\section{Denominator windows and the global schedule}

Section~3 shows that the LSST branching exponent is preserved for every
variable schedule satisfying the conditions stated there.  We now choose
such a schedule so that the prescribed lower envelope holds throughout
every return-time interval.

\subsection{Denominator localization}

The next lemma locates the third coordinate of any lattice vector whose
balance time lies between two consecutive return times \(t_n\) and
\(t_{n+1}\) and whose balanced minimum is too small.  Its hypotheses depend only on the level-\(n\) lattice
\(\Lambda_n(\kappa)\), not on the preceding branch.

\begin{lemma}\label{lem:denominator-window}
Let \(n\ge0\).  Suppose
\[
 \Lambda_n(\kappa)=a_{t_n}u_\kappa\Lambda^\circ
 \in\cL_3'\cap\cK_{\eps_n^2}^*(3),
 \qquad 0<\eps_n\le1,
 \qquad x\in\beta_n(\kappa).
\]
Let
\[
v=(p_1,p_2,q)\in\Lambda^\circ,\qquad
q>0,\qquad
R_x(v)>0,
\]
and let \(\alpha_n>0\).  Suppose
\[
\tau_x(v)\in[t_n,t_{n+1}],
\qquad
m_x(v)<\log\alpha_n.
\]
Then
\begin{equation}
12^{-1/w_2}\eps_n^{2/w_2}e^{t_n}
\le q
<
\alpha_ne^{t_{n+1}}.
\end{equation}
\end{lemma}

\begin{proof}
Define
\[
 \delta_i=e^{(1+w_i)t_n}(x_i-\kappa_i)
 \qquad(i=1,2),
\]
and write
\[
 a_{t_n}u_x\Lambda^\circ=
 \begin{pmatrix}
 1&0&-\delta_1\\
 0&1&-\delta_2\\
 0&0&1
 \end{pmatrix}
 \Lambda_n(\kappa).
\]
Since \(x\in\beta_n(\kappa)\),
\[
 (\delta_1^2+\delta_2^2)^{1/2}
 \le\sqrt2\,\eps_ne^{-w_2d_{n+1}}
 \le\sqrt2.
\]
The inverse of the displayed unipotent matrix therefore has operator
norm at most \(1+\sqrt2\).  Together with
\cref{lem:vertical-transference}, this gives
\begin{equation}\label{eq:node-systole}
 \lambda_1(B_1,a_{t_n}u_x\Lambda^\circ)
 \ge\frac{\eps_n^2}{2(1+\sqrt2)}.
\end{equation}

Put \(y=qe^{-t_n}\).  Since \(\tau_x(v)\ge t_n\),
\[
R_x(v)\le qe^{-2t_n}.
\]
Thus
\[
e^{w_it_n}\abs{qx_i-p_i}
\le
y^{w_i}.
\]
If \(0<y\le1\), then
\[
\norm{a_{t_n}u_xv}
\le
\sqrt{y^{2w_1}+y^{2w_2}+y^2}
\le
\sqrt3\,y^{w_2}.
\]
Together with \eqref{eq:node-systole}, this gives
\[
 y\ge
 \left(\frac{1}{2\sqrt3(1+\sqrt2)}\right)^{1/w_2}
 \eps_n^{2/w_2}
 \ge12^{-1/w_2}\eps_n^{2/w_2}.
\]
If \(y\ge1\), the last inequality follows directly from
\(0<\eps_n\le1\).  This proves the lower bound.

By the definitions of \(m_x(v)\) and \(\tau_x(v)\),
\[
 m_x(v)+\tau_x(v)=\log q.
\]
Consequently,
\[
 \log q<\log\alpha_n+t_{n+1},
\]
which proves the strict upper bound.
\end{proof}

\subsection{The global schedule}

The schedule must satisfy the variable-step tree hypotheses and
simultaneously make every denominator window empty.  The hierarchy
\(H^{1/4}\ll H^{1/2}\ll H\) as \(H\to\infty\), together with the
\(1\)-Lipschitz property of \(\psi\), makes these requirements
compatible; the auxiliary step from \(0\) to \(t_0\) constructs only
the root.

\begin{proposition}\label{prop:global-scheduling}
Let \(T_*\ge0\) and \(0<\eta\le1\).  Suppose that
\[
 \Lambda^\circ\in\cL_3'\cap\cK_{\eta^2}^*(3),
\]
and that
\[
 \psi:[T_*,\infty)\to(0,\infty)
\]
is nondecreasing, \(1\)-Lipschitz, and divergent.  For every
sufficiently large \(T\), with the threshold allowed to depend on
\(\psi\), set \(t_0=T\) and, for \(n\ge0\), define
\begin{align}
 H_n&:=\psi(t_n),\qquad
 d_{n+1}:=H_n^{1/2},\qquad
 t_{n+1}:=t_n+d_{n+1},
 \label{eq:global-schedule}\\
 \eps_n&:=e^{-H_n^{1/4}},\qquad
 \alpha_n:=e^{-H_n}.
 \label{eq:global-eps-alpha}
\end{align}
One may apply \cref{thm:child-count} once to the auxiliary parent
\((0,\Lambda^\circ)\) and choose a center \(\kappa_0\) such that
\[
 a_Tu_{\kappa_0}\Lambda^\circ
 \in\cL_3'\cap\cK_{\eps_0^2}^*(3).
\]
Then the recursion can be carried out at every level
\(n\ge1\), and every level \(n-1\) parent satisfies
\[
 c_-\eps_n^2e^{2d_n}
 \le \card{\mathcal C_n(\kappa)}
 \le c_+\eps_n^2e^{2d_n}
 \qquad(n\ge1).
\]
The resulting word tree is nested, nonempty, regular, and satisfies
the separation estimate of \cref{lem:separation}.

For every \(n\ge0\),
\begin{equation}\label{eq:global-empty-ineq}
 H_n
 >
 H_n^{1/2}+\frac2{w_2}H_n^{1/4}
 +\frac{\log12}{w_2}.
\end{equation}
Equivalently,
\[
 12^{-1/w_2}\eps_n^{2/w_2}e^{t_n}
 >
 \alpha_ne^{t_{n+1}},
\]
so the interval in \cref{lem:denominator-window} is empty.  Moreover,
\[
 \eps_n\downarrow0,
 \qquad d_n\to\infty,
 \qquad
 d_{n+1}=o(t_n),
 \qquad
 |\log\eps_n|=o(d_n).
\]
Consequently, all the hypotheses of
\cref{lem:singularity-shift,prop:tree-dimension} hold.
All fixed thresholds depend only on \(w\) and \(\eta\); the initial
time may also depend on \(\psi\).
\end{proposition}

\begin{proof}
Put
\[
 \gamma_w=\min\{w_2,w_1-w_2\}>0.
\]
Choose \(H_*=H_*(w,\eta)\) sufficiently large that, whenever
\(H\ge H_*\),
\[
\begin{gathered}
 H^{1/2}\ge d_*,
 \qquad
 e^{-H^{1/4}}\le\min\{\eta,\eps_*,r/8\},\\
 100\bigl(\log2+2^{1/4}H^{1/4}\bigr)
 \le\gamma_wH^{1/2},\\
 \frac12+e^{-(1+w_2)H^{1/2}}\le1,
 \qquad
 \log c_-+2H^{1/2}-2^{5/4}H^{1/4}\ge0,\\
 H>H^{1/2}+\frac2{w_2}H^{1/4}
 +\frac{\log12}{w_2}.
\end{gathered}
\]
Such a choice exists because the \(H^{1/4}\)-terms are
\(o(H^{1/2})\), and all terms on the right of the last inequality are
\(o(H)\).

We next choose the auxiliary seed time.  The \(1\)-Lipschitz property
gives, for \(T\ge T_*\),
\[
 \psi(T)
 \le
 \psi(T_*)+T-T_*
 =T+\psi(T_*)-T_*.
\]
Therefore,
\[
 \psi(T)^{1/4}
 \le\bigl(T+\psi(T_*)-T_*\bigr)^{1/4}
 =o(T).
\]
Since \(\psi(T)\to\infty\), we may take \(T\) so large that
\(\psi(T)\ge H_*\), \(T\ge d_*\), and
\begin{equation}\label{eq:seed-conditions}
 100\bigl(\log2+\psi(T)^{1/4}\bigr)\le\gamma_wT,
 \qquad
 \log c_-+2T-2\psi(T)^{1/4}\ge0.
\end{equation}
The choice of \(H_*\) also gives
\(e^{-\psi(T)^{1/4}}\le\min\{\eta,\eps_*\}\).  Thus the auxiliary parent,
which has time \(0\), center \(0\), lattice \(\Lambda^\circ\), and
parent scale \(\eta\),
satisfies all the hypotheses of \cref{thm:child-count}.  Indeed,
\[
 \Lambda^\circ\in\cL_3'\cap\cK_{\eta^2}^*(3),
\]
the first inequality in \eqref{eq:seed-conditions} is the counting
condition, and the second makes the lower child count at least one.
We may therefore choose a child \(\kappa_0\), for which
\[
 a_Tu_{\kappa_0}\Lambda^\circ
 \in\cL_3'\cap\cK_{\eps_0^2}^*(3)
\]
as required.  We now use \(\kappa_0\) as the center of the root at time
\(t_0=T\).

With \(t_0=T\), define the remaining parameters by
\eqref{eq:global-schedule}--\eqref{eq:global-eps-alpha}.
For every \(n\ge1\), monotonicity and the \(1\)-Lipschitz property
give
\begin{equation}
 0\le H_n-H_{n-1}
 \le t_n-t_{n-1}
 =d_n
 =H_{n-1}^{1/2}.
\end{equation}
Hence \(H_n\), \(d_{n+1}\), and \(H_n^{1/4}\) are nondecreasing,
while \(\eps_n\) is nonincreasing.  Since
\(H_{n-1}\ge H_0\ge H_*\ge1\),
\begin{equation}\label{eq:H-quarter-uniform}
 H_n^{1/4}
 \le
 (H_{n-1}+H_{n-1}^{1/2})^{1/4}
 \le
 2^{1/4}H_{n-1}^{1/4}
 =
 2^{1/4}d_n^{1/2}.
\end{equation}
Also \(d_n\ge d_1=H_0^{1/2}\).

At recursive level \(n\), the parent legality condition gives
\[
 a_{t_{n-1}}u_\kappa\Lambda^\circ
 \in\cL_3'\cap\cK_{\eps_{n-1}^2}^*(3).
\]
Thus \cref{thm:child-count} is applied with parent scale
\(\eta=\eps_{n-1}\); the inequality \(\eps_n\le\eps_{n-1}\) supplies
its scale hypothesis.

The choice of \(H_*\) and \eqref{eq:H-quarter-uniform} imply the
three recursive requirements
\[
\begin{gathered}
 d_n\ge d_*,
 \qquad
 \eps_n\le\min\{\eps_*,r/8\},\\
 100\bigl(\log2+|\log\eps_n|\bigr)\le\gamma_wd_n,
 \qquad
 \log c_-+2d_n-2|\log\eps_n|\ge0.
\end{gathered}
\]
Indeed, the scale condition in the choice of \(H_*\) gives
\[
 \eps_n\le\min\{\eps_*,r/8\},
\]
and monotonicity gives \(\eps_n\le\eps_{n-1}\).
Using \eqref{eq:H-quarter-uniform} and the counting condition in the
choice of \(H_*\),
\[
 100\bigl(\log2+|\log\eps_n|\bigr)
 =
 100\bigl(\log2+H_n^{1/4}\bigr)
 \le\gamma_wd_n,
\]
which is equivalent to the counting hypothesis of
\cref{thm:child-count}.

Nesting and separation are then automatic.  Since
\(d_{n+1}\ge d_n\), the choice of \(H_*\), applied with
\(H=H_{n-1}\), yields
\[
 \frac12+e^{-d_n-w_2d_{n+1}}
 \le
 \frac12+e^{-(1+w_2)d_n}
 \le1.
\]
Together with \(\eps_n\le\eps_{n-1}\), this proves
the nesting condition.  The same monotonicity gives
\[
 \eps_ne^{w_2(d_n-d_{n+1})}
 \le\eps_n\le r/8,
\]
which gives both separation requirements.

Finally, by \eqref{eq:H-quarter-uniform},
\begin{align*}
 \log\left(
 c_-\eps_n^2e^{2d_n}
 \right)
 &=
 \log c_-+2d_n-2H_n^{1/4}\\
 &\ge
 \log c_-
 +2d_n-2^{5/4}d_n^{1/2}
 \ge0
\end{align*}
by the choice of \(H_*\) with \(H=H_{n-1}\).  This is the
non-extinction requirement.  Moreover,
\[
 t_n=T+\sum_{j=1}^n d_j
 \ge T+nH_*^{1/2}\longrightarrow\infty.
\]
Consequently, the nesting argument and
\cref{thm:child-count,prop:regularity,lem:separation} apply at every
recursive level \(n\ge1\), including the first.

For \(n\ge0\), the last condition in the choice of \(H_*\), applied
to \(H=H_n\), gives
\[
 H_n
 >
 H_n^{1/2}
 +\frac2{w_2}H_n^{1/4}
 +\frac{\log12}{w_2},
\]
which is \eqref{eq:global-empty-ineq}.  Exponentiating the equivalent
logarithmic inequality gives
\[
 12^{-1/w_2}\eps_n^{2/w_2}e^{t_n}
 >
 \alpha_ne^{t_{n+1}},
\]
so the exact denominator interval is empty beginning with the first
return-time interval \([t_0,t_1]\).

It remains to prove the asymptotic assertions.  The Lipschitz bound
gives, for \(t_n\ge T_*\),
\[
 H_n=\psi(t_n)
 \le t_n+\psi(T_*)-T_*.
\]
Therefore
\[
 \frac{d_{n+1}}{t_n}
 =
 \frac{H_n^{1/2}}{t_n}
 \longrightarrow0.
\]
Finally, \eqref{eq:H-quarter-uniform} gives
\[
 \frac{|\log\eps_n|}{d_n}
 =
 \frac{H_n^{1/4}}{H_{n-1}^{1/2}}
 \le
 2^{1/4}H_{n-1}^{-1/4}
 \longrightarrow0.
\]
Since \(t_n\to\infty\) and \(\psi(t)\to\infty\), one also has
\(\eps_n\to0\) and \(d_n\to\infty\).  This proves the remaining
assertions.
\end{proof}

The global schedule now has both required properties:
it satisfies all hypotheses of the variable-step tree and makes every
denominator window empty.  The proof of the main theorem is therefore
a direct combination of Sections~2--4.

\section{Proof of the main theorem}

\begin{proof}[Proof of \cref{thm:main}]
Regularize \(f\) using \cref{lem:regularization}, obtaining a
nondecreasing, \(1\)-Lipschitz function \(\psi\) with
\(\psi(t)\to\infty\) such that
every \(x\in\Sing_w(2)\) satisfying
\[
W_x(t)\ge-\psi(t)
\quad\text{for all sufficiently large }t
\]
belongs to \(E_{f,w}\).
Since \(\psi(t)\to\infty\), choose \(T_*\ge0\) so that
\[
\psi(t)>0\qquad(t\ge T_*).
\]
Run the construction of Sections~3 and~4 with
\(\Lambda^\circ=\Z^3\) and \(\eta=1\).  Apply
\cref{prop:global-scheduling} to the restriction of \(\psi\) to
\([T_*,\infty)\), and choose \(T\), the root center \(\kappa_0\), and
all recursive parameters.  Define the rooted word tree and its limit
set \(F\) as in Section~3.  The scheduling proposition and
\cref{thm:child-count,prop:regularity,lem:separation}
show that this is a regular self-affine structure from its first
recursive level.  By
\cref{lem:singularity-shift,prop:tree-dimension},
\[
F\subseteq\Sing_w(2),
\qquad
\dim_H F=s_w.
\]
Every nonzero lattice vector with nonzero third coordinate determines
at most one \(x\) for which
\(u_x\Z^3\cap\R e_3\ne\{0\}\).  This exceptional set is therefore
countable.  Removing it does not change the dimension of \(F\).  Fix
\(x\in F\) outside this set, and let
\(v=(p_1,p_2,q)\in\wh{\Z^3}\) have \(q>0\) and
\(\tau_x(v)\ge t_0\).  The half-open intervals
\([t_n,t_{n+1})\), \(n\ge0\), partition \([t_0,\infty)\), so there
is a unique \(n\ge0\) with
\[
\tau_x(v)\in[t_n,t_{n+1}).
\]
If
\[
m_x(v)<-\psi(\tau_x(v)),
\]
then monotonicity gives
\[
m_x(v)<-\psi(t_n)=\log\alpha_n.
\]
The point \(x\) lies in its level \(n\) branch rectangle, so
\cref{lem:denominator-window} yields
\[
12^{-1/w_2}\eps_n^{2/w_2}e^{t_n}
\le q<\alpha_ne^{t_{n+1}}.
\]
This contradicts \cref{prop:global-scheduling}.  Therefore
\[
m_x(v)\ge-\psi(\tau_x(v))
\]
for every \(v=(p_1,p_2,q)\in\wh{\Z^3}\) with \(q>0\) and
\(\tau_x(v)\ge t_0\).  Hence
\cref{prop:eventual-balance} gives
\[
W_x(t)\ge-\psi(t)
\quad\text{for all sufficiently large }t.
\]
Together with \(x\in F\subseteq\Sing_w(2)\),
\cref{lem:regularization} gives \(x\in E_{f,w}\).  Since only a
countable subset of \(F\) was removed, the preceding inclusion and
dimension calculation give
\[
\dim_H E_{f,w}\ge s_w.
\]
The reverse inequality follows from
\(E_{f,w}\subseteq\Sing_w(2)\) and
\eqref{eq:LSST-dimension}.
\end{proof}

\section{Weighted uniform approximation}
\label{sec:uniform-approximation}

This section proves \cref{thm:subpower-uniform-rate}.  We first derive
from \cref{thm:main} a general rate-avoidance statement, which gives
the dimension assertion for the complement.  We then prove the
uniform-approximation assertion separately, using the Section~3 tree
with a power schedule and without the denominator-window argument of
Section~4.

For \(Q\ge1\), define
\[
 \Delta_x(Q)
 =\min_{\substack{v=(p_1,p_2,q)\in\Z^3\\1\le q\le Q}}R_x(v),
\]
and, for a positive function \(\Phi:[1,\infty)\to(0,\infty)\), let
\[
 \operatorname{UA}_w(\Phi)
 =\left\{
 x\in\R^2:
 \Delta_x(Q)\le\Phi(Q)
 \text{ for all sufficiently large }Q
 \right\}.
\]
Here \(\operatorname{UA}_w(\Phi)\) is the
\(m=2,n=1,\alpha=w,\beta=(1)\) case of KMWW's weighted definition
\cite[(1.12)]{KleinbockMoshchevitinWarrenWeiss}; see also their
unweighted uniform/singular implication
\cite[(1.3)]{KleinbockMoshchevitinWarrenWeiss}.
In this weighted setting, the same elementary argument gives
\[
 \Sing_w(2)
 =\left\{x\in\R^2:Q\Delta_x(Q)\longrightarrow0\right\},
 \qquad
 Q\Phi(Q)\longrightarrow0
 \Longrightarrow
 \operatorname{UA}_w(\Phi)\subseteq\Sing_w(2).
\]
Indeed, applying the uniform condition at height \(Q/2\) handles the
difference between \(q\le Q\) here and \(q<Q\) in the definition of
singularity, while \(Q\Phi(Q/2)\to0\) whenever \(Q\Phi(Q)\to0\).
We use the second implication below.

For the subpower rates in \cref{thm:subpower-uniform-rate}, we use
\[
 \Phi_{\nu,\mu}(Q)
 =Q^{-1}\exp\!\left(-\mu(\log Q)^\nu\right)
 \qquad(Q\ge e),
\]
with any positive nonincreasing extension to \([1,e]\).  The choice of
extension does not affect an eventual uniform approximation condition.

\begin{corollary}\label{cor:uniform-rate-avoidance}
Let \(\Phi:[1,\infty)\to(0,\infty)\) be nonincreasing and satisfy
\[
 Q\Phi(Q)\longrightarrow0.
\]
Then
\[
 \dim_H
 \left\{
 x\in\Sing_w(2):
 \Delta_x(Q)\ge2\Phi(Q)
 \text{ for all sufficiently large }Q
 \right\}
 =s_w.
\]
In particular,
\[
 \dim_H\left(\Sing_w(2)\setminus\operatorname{UA}_w(\Phi)\right)
 =s_w.
\]
\end{corollary}

\begin{proof}
Choose \(Q_0\ge1\) so large that \(\Phi(Q)<Q\) for \(Q\ge Q_0\),
and put
\[
 t(Q)=\frac12\log\frac{Q}{\Phi(Q)}\qquad(Q\ge Q_0).
\]
Since \(\Phi\) is nonincreasing, \(t(Q)\) is strictly increasing and
tends to infinity.  Define \(h\) on the range of \(t\) by
\[
 h(t(Q))=2\sqrt{Q\Phi(Q)},
\]
and put \(h(s)=e^{-s}\) elsewhere.  Then \(h>0\) and \(h(s)\to0\).
If \(x\in E_{h,w}\), then for every sufficiently large \(Q\), every
\(p\in\Z^2\), and every \(1\le q\le Q\),
\[
 e^{-t(Q)}q\le\sqrt{Q\Phi(Q)}=\frac12h(t(Q)).
\]
The definition of \(E_{h,w}\) therefore forces
\[
 \max_{i=1,2}\abs{qx_i-p_i}^{1/w_i}
 \ge e^{-t(Q)}h(t(Q))=2\Phi(Q).
\]
Thus \(\Delta_x(Q)\ge2\Phi(Q)\) eventually.  The first dimension
formula follows from \cref{thm:main} and \eqref{eq:LSST-dimension}, and
the second follows from the first.  The same argument works with any
fixed factor greater than \(1\).
\end{proof}

To obtain points satisfying a prescribed uniform upper
rate, we now extract one approximant from each return level and then
choose a power schedule adapted to that rate.

\begin{lemma}\label{lem:recorded-approximants}
Let \(F\) be a limit set obtained from the construction of Section~3
with \(\Lambda^\circ=\Z^3\).  For every \(x\in F\) and every
\(n\ge0\), there exists \((p_n,q_n)\in\Z^2\times\N\), with
\(p_n=(p_{n,1},p_{n,2})\), such that
\[
 \frac12e^{t_n}<q_n\le e^{t_n}
\]
and
\[
 \max_{i=1,2}\abs{q_nx_i-p_{n,i}}^{1/w_i}
 \le\eps_n^{1/w_1}e^{-t_{n+1}}.
\]
\end{lemma}

\begin{proof}
Choose a branch whose level-\(n\) center \(\kappa_n\) satisfies
\(x\in\beta_n(\kappa_n)\).  Since
\(\Lambda_n(\kappa_n)\in\cL_3'\), there is a primitive
\((p_n,q_n)\in\Z^2\times\N\) such that
\[
 a_{t_n}u_{\kappa_n}(p_n,q_n)
 =\ell_3\bigl(\Lambda_n(\kappa_n)\bigr)e_3.
\]
Hence
\[
 q_n=\ell_3\bigl(\Lambda_n(\kappa_n)\bigr)e^{t_n},
\]
which gives the denominator bounds, and the first two coordinates give
\(p_{n,i}=q_n\kappa_{n,i}\).  Since \(x\in\beta_n(\kappa_n)\),
\[
 \abs{x_i-\kappa_{n,i}}
 \le\eps_ne^{-w_it_{n+1}-t_n},
\]
so
\[
 \abs{q_nx_i-p_{n,i}}
 \le\eps_ne^{-w_it_{n+1}}.
\]
Taking \(1/w_i\)-powers and using \(0<\eps_n\le1\) and \(w_1>w_2\)
proves the claim.
\end{proof}

\begin{lemma}\label{lem:power-schedule}
Let \(0<\nu<1\) and \(\chi>0\), and put
\[
 b_\nu=\frac{1+\nu}{2}.
\]
For every sufficiently large \(t_0\), the construction of Section~3
with \(\Lambda^\circ=\Z^3\) can be initialized with
\[
 d_{n+1}=t_n^{b_\nu},\qquad
 t_{n+1}=t_n+d_{n+1},\qquad
 \eps_n=e^{-\chi t_n^\nu}.
\]
Its limit set \(F\) satisfies
\[
 F\subseteq\Sing_w(2),\qquad \dim_H F=s_w.
\]
\end{lemma}

\begin{proof}
Let \(\gamma_w=\min\{w_2,w_1-w_2\}\).  Since
\(\nu<b_\nu<1\), choose \(t_0\ge1\) so large that
\[
 t_0\ge d_*,\qquad t_0^{b_\nu}\ge d_*,\qquad
 \eps_0=e^{-\chi t_0^\nu}\le\min\{1,\eps_*\},
\]
\[
 100(\log2+\chi t_0^\nu)\le\gamma_wt_0,
 \qquad
 \log c_-+2t_0-2\chi t_0^\nu\ge0.
\]
Require, in addition, that for every \(s\ge t_0\),
\[
 100\bigl(\log2+2^\nu\chi s^\nu\bigr)
 \le\gamma_ws^{b_\nu},
 \qquad
 \log c_-+2s^{b_\nu}-2^{1+\nu}\chi s^\nu\ge0,
\]
\[
 \frac12+e^{-(1+w_2)s^{b_\nu}}\le1,
 \qquad
 e^{-\chi s^\nu}\le\min\{\eps_*,r/8\}.
\]
These requirements hold for every sufficiently large \(t_0\) because
\(\nu<b_\nu<1\).  Since
\(\Z^3\in\cL_3'\cap\cK_1^*(3)\), \cref{thm:child-count}, applied to
the auxiliary step from time \(0\) to \(t_0\) with parent scale \(1\),
now produces a root center \(\kappa_0\) such that
\[
 a_{t_0}u_{\kappa_0}\Z^3
 \in\cL_3'\cap\cK_{\eps_0^2}^*(3).
\]
At level \(n\ge1\), parent legality gives
\[
 a_{t_{n-1}}u_\kappa\Z^3
 \in\cL_3'\cap\cK_{\eps_{n-1}^2}^*(3),
\]
so \cref{thm:child-count} is used with parent scale
\(\eta=\eps_{n-1}\).  Moreover,
\[
 \frac{t_n}{t_{n-1}}
 =1+t_{n-1}^{b_\nu-1}\le2,
\]
so
\[
 |\log\eps_n|=\chi t_n^\nu
 \le2^\nu\chi t_{n-1}^\nu=o(d_n),
 \qquad d_n=t_{n-1}^{b_\nu}.
\]
The displayed choices therefore imply the child-count and
non-extinction hypotheses at every level.  Since \(d_{n+1}\ge d_n\)
and \(\eps_n\le\eps_{n-1}\), they also imply the nesting and
separation conditions in
\cref{prop:regularity,lem:separation}.  Finally,
\[
 \frac{d_{n+1}}{t_n}=t_n^{b_\nu-1}\longrightarrow0,
 \qquad
 \frac{|\log\eps_n|}{d_n}\longrightarrow0,
\]
and \(d_n\to\infty\), \(\eps_n\to0\).  Hence
\cref{lem:singularity-shift,prop:tree-dimension} give
\(F\subseteq\Sing_w(2)\) and \(\dim_H F=s_w\).
\end{proof}

\begin{proof}[Proof of \cref{thm:subpower-uniform-rate}]
Choose \(\chi>w_1\mu\), and let \(F\) be the limit set supplied by
\cref{lem:power-schedule}.  Fix \(x\in F\).  If
\(e^{t_n}\le Q<e^{t_{n+1}}\), then
\cref{lem:recorded-approximants} gives \(q_n\le Q\) and
\[
 \Delta_x(Q)
 \le Q^{-1}\exp\!\left(-\frac{\chi}{w_1}t_n^\nu\right).
\]
Moreover,
\[
 \frac{t_{n+1}}{t_n}=1+t_n^{b_\nu-1}\longrightarrow1.
\]
Since \(\chi/w_1>\mu\), for all sufficiently large \(n\), uniformly
for \(e^{t_n}\le Q<e^{t_{n+1}}\),
\[
 \frac{\chi}{w_1}t_n^\nu\ge\mu(\log Q)^\nu.
\]
Thus \(F\subseteq\operatorname{UA}_w(\Phi_{\nu,\mu})\), and
\(\dim_H\operatorname{UA}_w(\Phi_{\nu,\mu})\ge s_w\).

Since
\[
 Q\Phi_{\nu,\mu}(Q)
 =\exp\!\left(-\mu(\log Q)^\nu\right)\longrightarrow0,
\]
the inclusion recorded at the beginning of this section gives
\(\operatorname{UA}_w(\Phi_{\nu,\mu})\subseteq\Sing_w(2)\).
Together with \eqref{eq:LSST-dimension}, this gives the reverse
dimension bound.  The dimension of the complement follows from
\cref{cor:uniform-rate-avoidance}.
\end{proof}

\begin{remark}
The endpoint \(\nu=1\) in \cref{thm:subpower-uniform-rate} is not
covered by the present construction.
\end{remark}

\section*{Acknowledgements}

The author thanks his advisor, Yitwah Cheung, for suggesting this
problem, and both Yitwah Cheung and Chengyang Wu for helpful
discussions.

\section*{Declaration of generative AI and AI-assisted technologies\\
in the manuscript preparation process}

During the preparation of this work, the author used ChatGPT and Codex
(OpenAI) to assist with proof auditing, manuscript organization,
bibliographic checks, and language editing.  After using these tools,
the author independently verified all mathematical arguments and
references and reviewed and edited the manuscript as needed.  The
author takes full responsibility for the content of the article.

\end{document}